\documentclass[11pt]{article}

\include{psfig}
\usepackage{pictex, latexsym, graphicx,amsmath,amssymb,amsbsy,amsfonts,amsthm,verbatim}
\usepackage[pdfpagemode=UseOutlines,colorlinks=false,pdfnewwindow=true,urlcolor=red]{hyperref}
\usepackage{color}
\usepackage{subfigure}
\usepackage{graphics}
\usepackage{multirow}

\usepackage{algorithm}

\newcommand{\beq}{\begin{equation}}
\newcommand{\eeq}{\end{equation}}
\newcommand{\bea}{\begin{eqnarray*}}
\newcommand{\eea}{\end{eqnarray*}}

\newtheorem{lemma}{Lemma}[section]
\newtheorem{theorem}[lemma]{Theorem}

\newtheorem{corollary}[lemma]{Corollary}

\newtheorem{observation}[lemma]{Observation}

\theoremstyle{definition}
\newtheorem{definition}[lemma]{Definition}

\input colordvi
\newcommand{\be}{\begin{equation}}
\newcommand{\ee}{\end{equation}}
\newcommand{\bx}{\begin{bmatrix}}
\newcommand{\ex}{\end{bmatrix}}

\newtheorem{defn}{Definition}[section]

\newtheorem{cor}[defn]{Corollary}

\newtheorem{thm}[defn]{Theorem}

\theoremstyle{definition}

 \newenvironment{cit}
{
    \begin{list}{- \ }{}
        \setlength{\topsep}{0pt}
        \setlength{\parskip}{0pt}
        \setlength{\partopsep}{0pt}
        \setlength{\parsep}{0pt}         
        \setlength{\itemsep}{0pt} 
}
{
    \end{list} 
}
 \newenvironment{cem}
{
    \begin{enumerate}
        \setlength{\topsep}{0pt}
        \setlength{\parskip}{0pt}
        \setlength{\partopsep}{0pt}
        \setlength{\parsep}{0pt}         
        \setlength{\itemsep}{0pt} 
}
{
    \end{enumerate} 
}

\begin{document}

\title{On the hardness of recognizing triangular line graphs}
\author{Pranav Anand \\
                Department of Linguistics\\
                University of California--Santa Cruz\\
                \texttt{panand@ucsc.edu}
                \and
                Henry Escuadro\\
                Department of Mathematics\\
                Juniata College\\
                \texttt{escuadro@juniata.edu}
                \and
                Ralucca Gera\\
                Department of Applied Mathematics\\
                Naval Postgraduate School\\
                \texttt{rgera@nps.edu}
                \and
                Stephen G. Hartke\thanks{This author is supported in part by a Nebraska EPSCoR First Award and NSF grant DMS-0914815.}\\
                Department of Mathematics\\
                University of Nebraska--Lincoln\\
                \texttt{hartke@math.unl.edu}
                \and
                Derrick Stolee\thanks{This author is supported in part by NSF grants DMS-0914815 and CCF-0830730.}\\
                Department of Computer Science\\
                Department of Mathematics\\
                University of Nebraska--Lincoln\\
                \texttt{dstolee@cse.unl.edu}
        }
\maketitle
\begin{abstract}
	Given a graph $G$, its triangular line graph
		is the graph $T(G)$ with vertex set 
		consisting of the edges of $G$
		and adjacencies between edges 
		that are incident in $G$
		as well as being within a common triangle.
	Graphs with a representation as the triangular line graph of some graph $G$
		are \emph{triangular line graphs}, which have been studied under many names
		including anti-Gallai graphs, 2-in-3 graphs, and link graphs.
	While closely related to line graphs, triangular line graphs
		have been difficult to understand and characterize.
	Van Bang Le asked if recognizing triangular line graphs has an efficient
		algorithm or is computationally complex.
	We answer this question by proving that 
		the complexity of recognizing 
		triangular line graphs is 
		NP-complete 
		via a reduction from 3-SAT.
\end{abstract}





\bigskip

\def\calG{{\mathcal G}}

\clearpage

\section{Introduction}

The line graph operator, $L$, is well studied.
Recognition and characterization have been determined for $L$ 
	(see Lehot~\cite{lehot74} and Roussopoulos~\cite{roussopoulos73}), 
	but these problems for related operators are less understood. 
The triangular line graph operator, $T$, 
	defined by Jarrett~\cite{OnIteratedTriangularLineGraphs} 
	is one such operator.
	
\begin{definition}[Triangular Line Graph]
	Given a graph $G$ with edge set $E(G)$, 
		the \emph{triangular line graph} of $G$, 
		denoted $T(G)$,
		is the graph with vertex set $E(G)$
		and an edge between the edges $e_1$ and $e_2$ if 
		and only if
	\begin{cit}
		\item $e_1$ and $e_2$ are incident at a common vertex in $G$, and
		\item there is an edge $e_3 \in E(G)$ so that $e_1e_2e_3$ is a triangle in $G$.
	\end{cit}
	
	A graph $H$ is a \emph{triangular line graph} if there exists a graph $G$
		so that $H \cong T(G)$.
\end{definition}
	
The second restriction is the only difference from the line graph operator, 
	but causes great difference in behavior.
For example, the class of triangular line graphs is not closed under
	taking induced subgraphs, unlike line graphs.

The triangular line graph is studied under many different names.
For integers $k$ and $\ell$, the $k$-in-$\ell$ operator $\Phi_{k,\ell}$ maps
	$k$-cliques to vertices which are adjacent if and only if
	they appear in a common $\ell$-clique \cite{GraphDynamics}.
For $k = 2$ and $\ell = 3$,  the $2$-in-$3$ operator $\Phi_{2,3}$ is equivalent
	to the triangular line graph operator.
For $\ell = k+1$, the $k$-in-$\ell$ operator 
	is identical to the $k$-anti-Gallai operator \cite{GraphDynamics}.
The name comes from the $k$-Gallai graph operator, where $k$-cliques are mapped to vertices 
	and adjacencies exist when the $k$-cliques intersect at $k-1$ vertices but 
	are not within a $k+1$-clique.
These operators were originally
	defined by Gallai~\cite{gallai} in connection with comparability graphs.
In addition, $T(G)$ is a special case of the $H$-line graph (when $H = K_3$), 
	defined by Chartrand, Gavlas, and Schultz~\cite{ConvergentSequencesOfIteratedHLineGraphs} 
	as a subgraph of $L(G)$ which preserves edge adjacency for edges in an induced subgraph 
	isomorphic to $H$. 
	
Prior research on triangular line graphs
 	has focused primarily on the properties of iterations of the operator 
	(see Jarret~\cite{OnIteratedTriangularLineGraphs}, 
	Dorrough~\cite{ConvergenceOfSequencesOfIteratedTriangularLineGraphs}, 
	and Anand, Escuadro, Gera, Martell~\cite{aegm1}, ~\cite{aegm2}), 
	their relation to perfect graphs \cite{GallaiGraphsAndAntiGallaiGraphs}, 
	or the forbidden subgraphs characterization for $H$-free triangular line graphs 
		\cite{GallaiAndAntiGallaiGraphsOfAGraph}.   

In \cite{GallaiGraphsAndAntiGallaiGraphs}, 
	Le provides a characterization for  $k$-Gallai graphs and $k$-anti-Gallai graphs,
	inspired by the characterization of line graphs in \cite{krausz43}. 
With $k=2$ this yields a characterization for triangular line graphs:

\begin{theorem}[Van Bang Le, \cite{GallaiGraphsAndAntiGallaiGraphs}]
	A graph $G$ is a triangular line graph if and only if there is a family of induced subgraphs 
		$\mathcal{F} = \{G_i: i \in I\}$ of $G$ such that:
		
	\begin{cem} 
	
		\item Every vertex in $G$ appears in exactly two members of $\mathcal{F}$.
	
		\item Every edge in $G$ appears in exactly one member of $\mathcal{F}$.
		
		\item $|G_i \cap G_j| \le 1$, for every $ i \neq j$. 
		
		\item For any distinct $i,j,k \in I$, if $\{v_i\} = G_i \cap G_k$ and $\{v_j\} = G_j \cap G_k$ and $v_i \ne v_j$, then $(v_i, v_j) \in E(G)$  if and only if $G_i \cap G_j \ne \emptyset$. 
	
	\end{cem}
\end{theorem}
	
This characterization arises when considering a \emph{triangular line graph preimage} $H$ of a graph $G$ so that
	$T(H) \cong G$.
For each vertex $v$ in $V(H)$, the edges incident to $v$ induce one of the subgraphs
	in the family $\mathcal{F}$.
While this is a complete characterization, 
	it is unsatisfying as it is cumbersome and lacks an efficient algorithmic interpretation.
In \cite{GallaiGraphsAndAntiGallaiGraphs}, Le states that the complexity of recognizing this class 
	is unknown.

We prove that recognition of triangular line graphs is NP-complete.  
We do so via a reduction from the canonical NP-complete problem $3$-SAT 
	(see \cite{GareyJohnson}). 
The logic of a conjunctive normal formula $\phi$
	with clauses of size three
	is encoded to a graph $G_\phi$
	whose triangular line graph preimage (if it exists)
	describes a satisfying assignment of $\phi$.
This result is in stark contrast to the polynomial-time algorithm of Roussopoulos~\cite{roussopoulos73}
	that recognizes line graphs and constructs the unique line graph preimage.
Moreover, it implies 
	simpler characterizations 
	of triangular line graphs
	that lead to polynomial-time algorithms
	 are extremely unlikely.

Our reduction relies on the fact that certain triangular line graphs have 
precisely two preimage isomorphism classes, thus providing us with a representation of binary values.
We begin with a discussion of these classes, and proceed to show how these classes enforce similar requirements
for the preimages of graphs containing these components as subgraphs.
From these two observations, we construct gadgets that encode 
	logic of 3-CNF formulas, starting 
	with binary values for each variable.
These gadgets are combined in the final construction that requires each 
	clause to be satisfied if and only if a preimage exists.


\section{Triangle-Induced Subgraph Closure}

An important property of line graphs is that any induced subgraph is also a line graph.
Triangular line graphs are not closed under induced subgraphs, but do satisfy a similar property.

\begin{definition}
	A subgraph $H_1$ is a \emph{triangle-induced subgraph} of $H_2$, denoted $H_1 \lhd H_2$,
		if 
	
	\begin{cem}
		\item $H_1$ is an induced subgraph of $H_2$ and
		\item If there is an edge $uv$ in $H_1$ and a triangle $uvw$ in $H_2$,
				then the vertex $w$ is in $H_1$.
	\end{cem}
\end{definition}

The interesting property of this relation is that
	any triangle-induced subgraph of a triangular line graph 
	is also a triangular line graph.

\begin{lemma}\label{lma:TriangleInducedClosure}
	Let $H_1, H_2$ be triangular line graphs, 
		with triangular line graph preimages and edge-to-vertex bijections given by 
	\[\mathcal{G}_j = \{ (G,f) : T(G) \cong H_j\text{ via the bijection } f : E(G) \to V(H_j) \}, 
		\text{ for }j \in \{1,2\}.\]	
	If $H_1 \lhd H_2$, then for any preimage $(G_2,f_2) \in \mathcal{G}_2$, there exists
		a preimage $(G_1,f_1) \in \mathcal{G}_1$ so that
		$G_1$ is isomorphic to a subgraph of $G_2$ 
		using the edge map $f_2^{-1} \circ f_1 : E(G_1) \to E(G_2)$.
\end{lemma}

\begin{proof}
	Let $(G_2,f_2) \in \mathcal{G}_2$ be a preimage for a triangular line graph $H_2$.
	{Set $G_1$ to be the graph induced by edge set 
		$F = f_2^{-1}(V(H_1))$}.
	It is clear that the adjacencies in $T(G_1) \subseteq H_2$ 
		are also adjacencies in {$T(G_2) = H_2$},
		as triangles in $G_1$ are triangles in $G_2$.
	Let $e_1$ and $e_2$ be edges in $G_1$ that are adjacent vertices in $H_1$.
	There is an edge $e_3$ in $G_2$ that completes the triangle with $e_1$ and $e_2$.
	Since $e_1$ and $e_2$ are vertices in $H_1$, 
		and $H_1$ is a triangle-induced subgraph of $H_2$,
		then $e_3$ must also be in $H_1$, 
		as these {vertices} form a triangle in $H_2$.
	Hence, $e_3$ is in $G_1$ and $e_1$ and $e_2$ are adjacent in $T(G_1)$.
	To complete the proof, define $f_1 = f_2|_{E(G_1)}$.
\end{proof}

It follows from this lemma that the class of triangular line graphs 
	is closed under taking triangle-induced subgraphs.

\begin{corollary}
	Let $G$ be any graph and $H = T(G)$ its triangular line graph.
	Any triangle-induced subgraph $H' \lhd H$
		is also a triangular line graph
		with preimage $G'$ that is a subgraph of $G$.
\end{corollary}

Given a graph $G$ and its line graph $H = T(G)$, we use the notation
	$T^{-1}(H')$ on any triangle-induced subgraph $H' \lhd H$
	to denote the edges in $G$ that map to the vertices of $H'$.
We proceed to define gadgets which become triangle-induced subgraphs of the final graph.
By {putting} restrictions on the class of preimages of each gadget,
	we prove that the larger graph has a preimage if and only if
	each gadget satisfies their individual preimage restrictions.
Put together, these gadgets encode the logic of a satisfied 3-CNF formula.


\section{Cycles, Wheels, and Suns}\label{subsection:7sun}

This section describes a class of triangular line graphs with multiple preimages.
The smallest unit of these graphs is the \emph{bowtie}: 
	the graph consisting of	two triangles intersecting at a single vertex, 
	called the \emph{center}.
The bowtie is the triangular line graph of the \emph{double triangle}: 
	{$K_4-e$}. 
While this is the only preimage (up to isomorphism), there are two different
	classes of labeled preimages, depending on the incidences of 
	edges between the triangles.

\begin{figure}[h]
\centering
	\mbox{
		\subfigure[The bowtie.]{
			\includegraphics[height=0.95in]{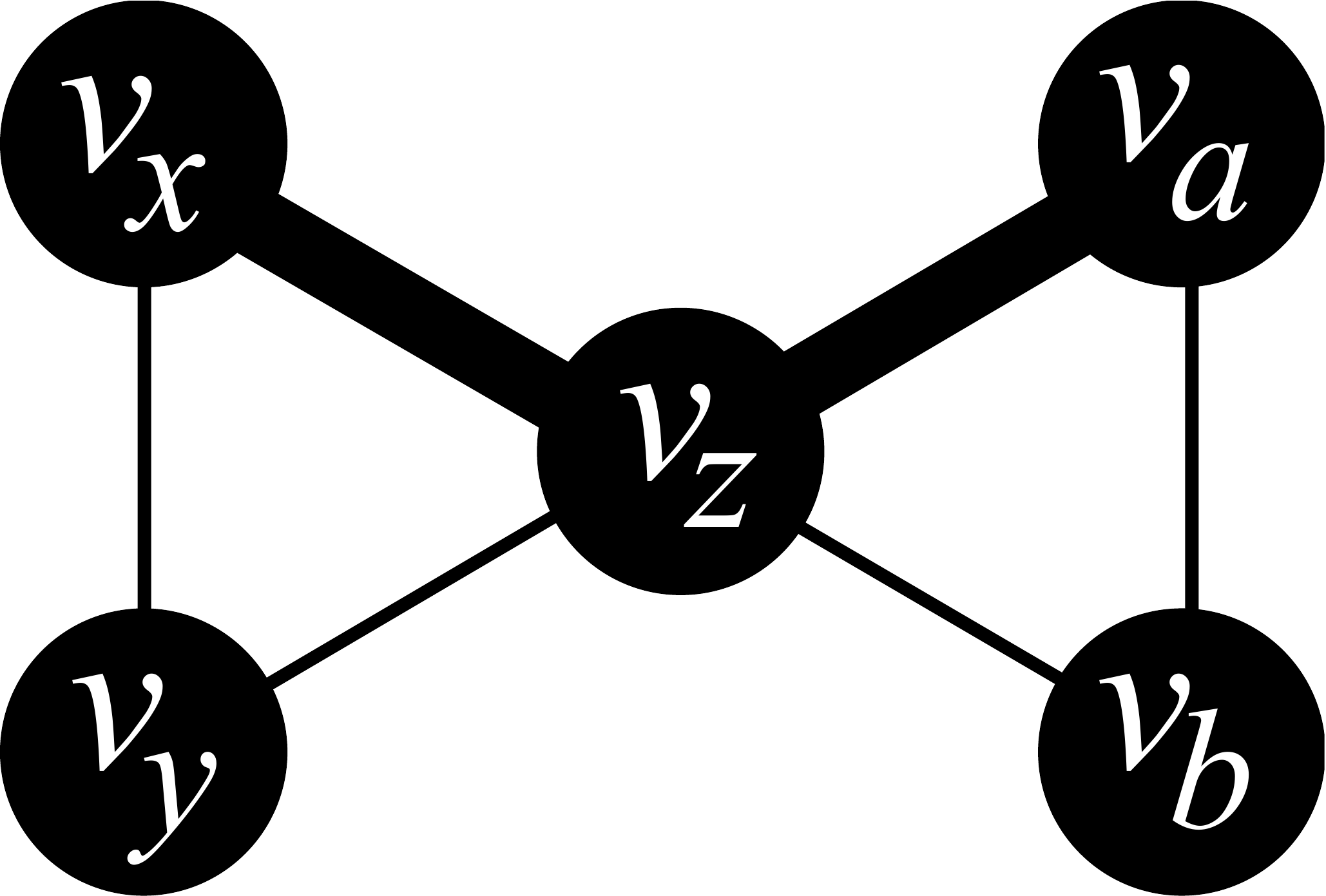}
		}
		\qquad
		\subfigure[First preimage of the bowtie.]{
			\includegraphics[height=0.95in]{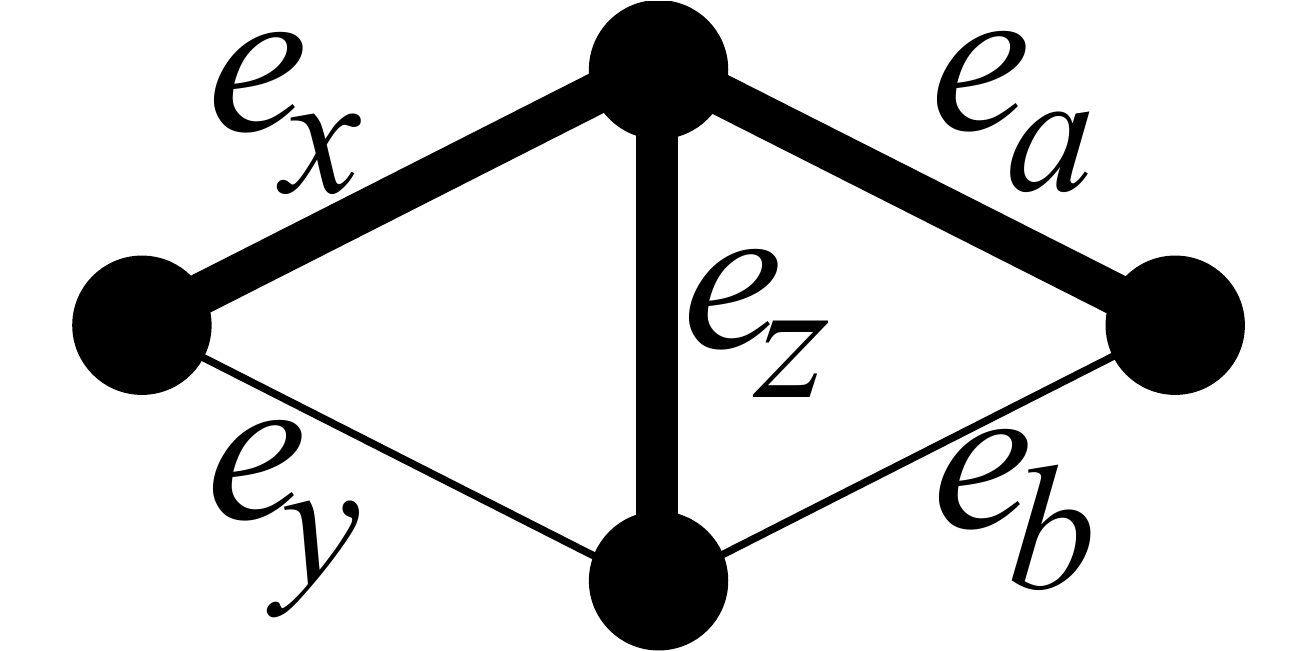}
		}
		\quad
		\subfigure[Second preimage of the bowtie.]{
			\includegraphics[height=0.95in]{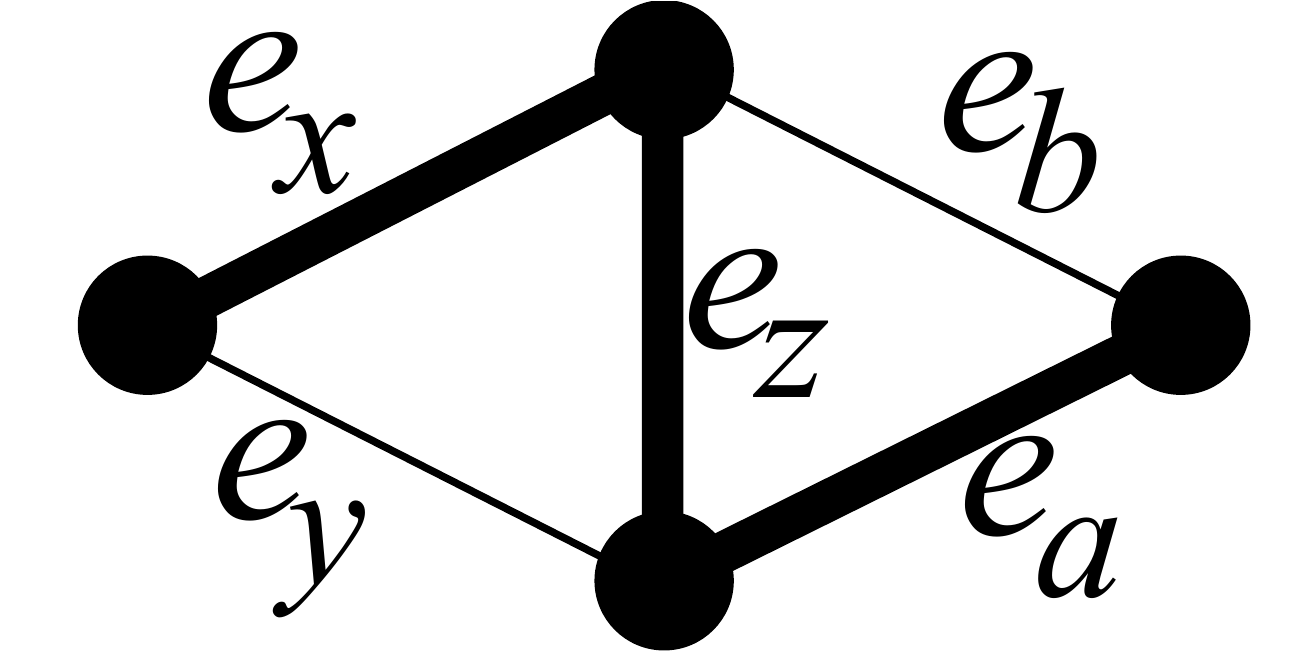}
		}
	}
	\caption{\label{fig:bowtiedbltri}The bowtie and its preimages with three highlighted edges.}
\end{figure}

Consider two nonadjacent vertices in the bowtie.
They are both adjacent to the center, 
	so their edges in the preimage 
	are adjacent to the center edge.
The two preimages differ depending on whether or not these 
	edges are incident to each other.
See Figure \ref{fig:bowtiedbltri} which shows these preimages.

A \emph{triangle trail} is a sequence of triangles where consecutive triangles
	intersect at an edge.
There are two special triangle trails of length $k$, for $k \geq 3$.
The first is the $k$-fan, where all triangles have a common vertex, 
	but are otherwise
	disjoint other than the edge between consecutive triangles.
The second is the $k$-triangle strip, where triangles of distance
	two intersect at a vertex, but triangles of larger distance are disjoint.
For $k=3$, these definitions coincide, and Figure \ref{fig:fans} 
	demonstrates the difference at $k = 4$.
	
\begin{figure}[ht]
	\centering
	\mbox{
		\subfigure[\label{fig:threefan}Three consecutive triangles make a fan.]{
			\includegraphics[height=1in]{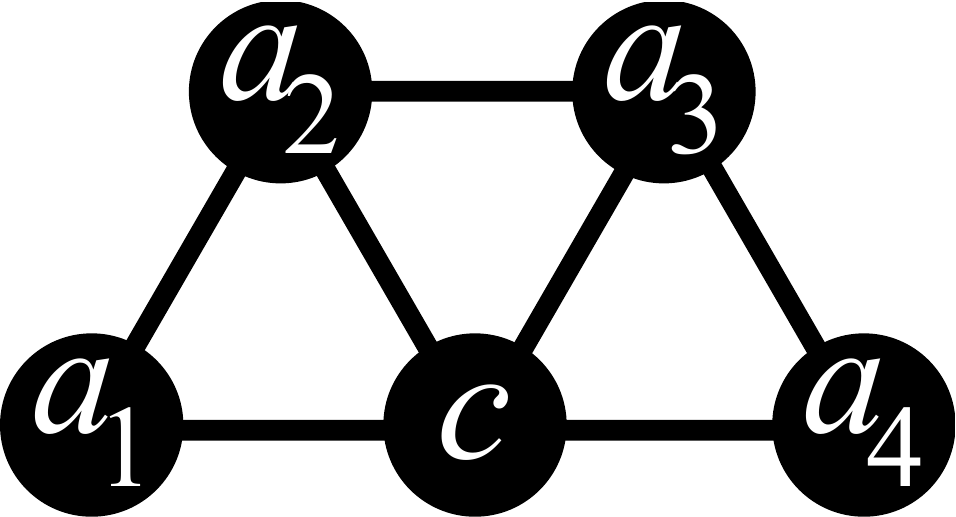}
			}
		\quad
		\subfigure[\label{fig:fourfan}Four triangles in a fan.]{
			\includegraphics[height=1.25in]{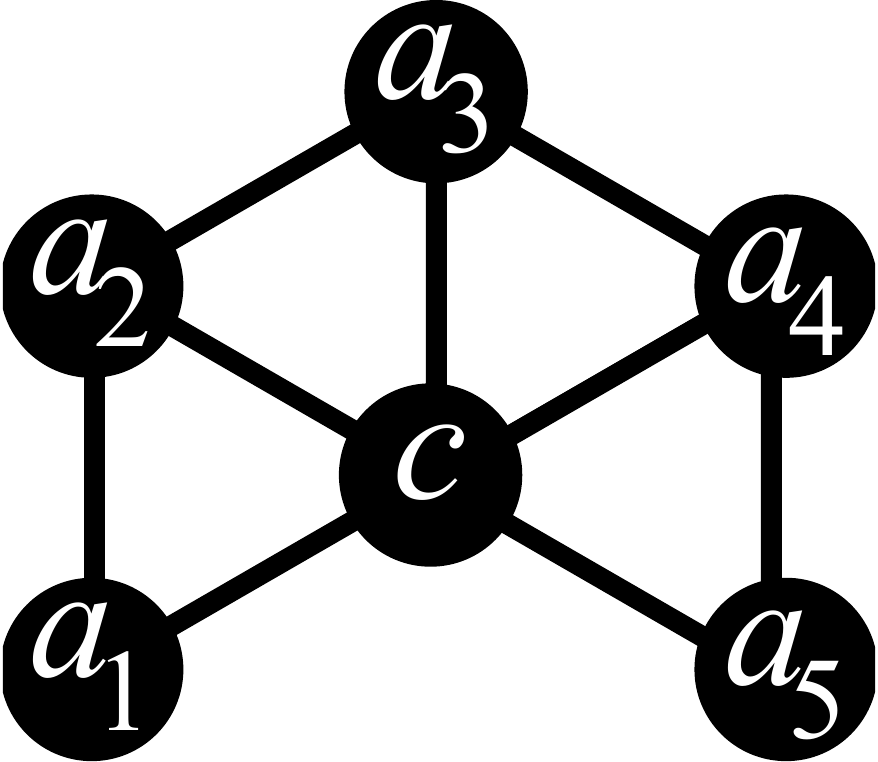}
			}
		\quad
		\subfigure[\label{fig:fourstrip}Four triangles in a strip]{
			\includegraphics[height=1in]{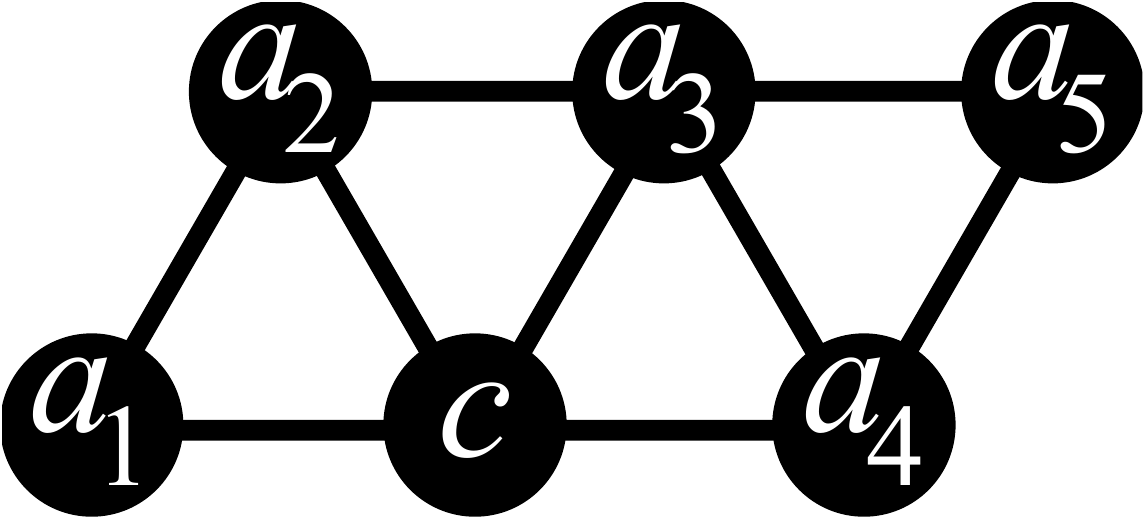}
			}
		
	}
	\caption{\label{fig:fans}Fans and triangle strips.}
\end{figure}

Notice that the $k$-triangle strip and $k$-fan have isomorphic triangular line graphs
	composed of a sequence of $k-1$ bowties where consecutive bowties intersect at a triangle
	and are otherwise disjoint.
By closing these sequences into cyclic orders, we form the base of our gadget constructions.

\begin{definition} Let $k \geq 4$ be an integer.
	\begin{cit}
		\item The \emph{$k$-wheel}, $W_k$, is the graph formed by a $k$-cycle and a dominating vertex, 
			called the \emph{center} of the wheel.
			Edges incident to the center are called \emph{spokes}.
			
		\item The \emph{squared $k$-cycle}, $C_{k}^2$, is the graph given by a $k$-cycle 
				with additional edges added for vertex pairs at distance two in the $k$-cycle.
				
		\item The \emph{$k$-sun}, $S_k$, is the graph given by a $k$-cycle with
				a new vertex for each edge, adjacent to the endpoints.
	\end{cit}
\end{definition}

\begin{lemma}\label{lma:7sun}
	For $k \geq 7$, the triangular line graphs of the $k$-wheel $W_k$ and squared cycle $C_k^2$
		are both isomorphic to the $k$-sun $S_k$.
	The 7-wheel $W_7$ and squared cycle $C_7^2$ are the only two preimages
		of the $7$-sun $S_7$ up to isomorphism.
\end{lemma}

\begin{proof}
	The first part of the lemma is easy to verify.
	
	Let $G$ be a preimage of $S_7$.  As mentioned above, $G$ is composed of seven triangles where each triangle intersects exactly two triangles at two distinct shared edges. Note that $G$ cannot contain an $n-$wheel where $n \le 6$. Moreover,  since $T(G) \cong S_7$ has exactly $14$
	vertices, it follows that $G$  has $14$ edges.  We consider two cases:
	
\begin{enumerate}
	
	\item[Case 1:] 
		$G$ does not contain a $4$-fan.
	\item[]	
		It follows that each set of four consecutive triangles in $G$ forms an 
			alternating strip of triangles, as in Figure \ref{fig:fourstrip}.
		Put together, we find that all seven triangles can be placed in a strip
			of seven alternating triangles, given by 
			vertices $v_1,v_2,\dots,v_9$, where
			$v_i$ is adjacent to $v_{i-2}, v_{i-1},v_{i+1},v_{i+2}$.
		The edges $v_1v_2$ and $v_8v_9$ must be identified to close the 
			cycle in the 7-sun.
		If $v_2$ and $v_8$ are identified, the
			vertices $v_2v_4v_6$ form a triangle
			which is not represented in the 7-sun.
		Hence, $v_1$ and $v_8$ are identified and 
			the vertices form a squared cycle $C_7^2$.
			
	\item[Case 2:] $G$ contains a $4$-fan. 
	
	\item[]
	It follows that there are four consecutive triangles in $G$ that share a common vertex $c.$
	Let $a_1,\dots,a_5$ be the endpoints
		around this $4$-fan centered at $c$ as in Figure \ref{fig:fourfan}.
	Since the triangles $\langle a_1, c, a_2\rangle$ and $\langle a_4, c, a_5\rangle$ intersect other triangles
		at the edges $ca_2$ and $ca_4$, respectively,
		these edges cannot be used to intersect $\langle a_1, c, a_2\rangle$ and $\langle a_4, c, a_5\rangle$ with other adjacent triangles. We consider two cases according to which edges of $\langle a_1, c, a_2\rangle$ and $\langle a_4, c, a_5\rangle$ are shared with neighboring triangles (different from the triangles that are in the $4$-fan shown in Figure~\ref{fig:fourfan}).
	
\begin{enumerate}
		\item[Case 2.1:] Both edges $ca_1$ and $ca_5$ are shared with neighboring triangles. 
		
		\item[]
		We have three possibilities.
		If $a_1 a_5 \in E(G)$, then $G$ contains a \emph{5-wheel} which is impossible.
		If $a_1, a_5$ and $c$ are adjacent to some new vertex $x$, then  $G$ contains a \emph{6-wheel} which is impossible.
		Lastly, if $a_1$ and $c$ are adjacent to some new vertex $x$, and $a_5$ and $c$ are adjacent to some new vertex $y \neq x$, then we have accounted for $13$ edges of $G$ so far. Since $G$ has $14$ edges, it follows that $xy \in E(G)$ and so $G \cong W_7.$
	\item[Case 2.2:] Only one of the edges $ca_1$ and $ca_5$ is shared with neighboring triangles.
	
	\item[] Without loss of generality, let the edge $ca_1$ be shared with a neighboring triangle of $\langle a_1, c, a_2\rangle.$ It follows that the edge $a_4a_5$ must be shared with a neighboring triangle of $\langle a_4, c, a_5\rangle$.
	Since no edge $a_i a_j$ $(1 \le i \neq j \le 5)$ can be introduced as we would obtain a $W_n$ with $n \le 6$, it follows that both $a_1$ and $c$ must adjacent to a new vertex $x$ and 
	both $a_4$ and $a_5$ must be adjacent to a new vertex $y\ne x.$ Since we have already accounted for $13$ of the $14$ edges of $G$, we do not have enough edges left to form the other neighboring triangles of $\langle a_1, c, x\rangle$ and $\langle a_4, a_5, y\rangle.$ \qedhere

%
	
%
%
%
\end{enumerate}	

\end{enumerate}
\end{proof}

It is necessary for later arguments to understand some structure of the squared cycle $C_{k}^2$.

\begin{observation}\label{obs:squaredcycleembed}
	Let $k \geq 7$ be an integer.
	
	\begin{cit}
		\item If $k$ is even, then $C_{k}^2$ is embeddable on the cylinder.
		
		\item If $k$ is odd, then $C_{k}^2$ is embeddable on the M\"obius strip.
	\end{cit}
\end{observation}

The M\"obius strip embedding of $C_{7}^2$ is given below in Figure \ref{fig:mobiustrianglestrip}.
	
\begin{figure}[ht]
	\centering		
			\includegraphics[width=0.6\textwidth]{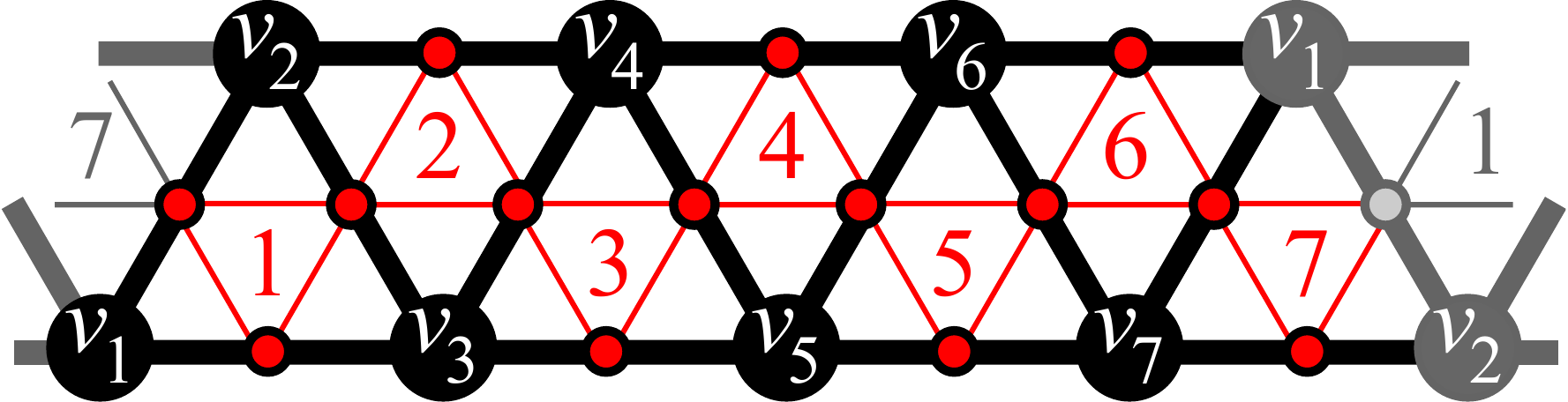}		
	\caption{\label{fig:mobiustrianglestrip}Embedding $C_{7}^2$ on the M\"obius strip.
		The triangular line graph $S_7$ is shown in red.}
\end{figure}


The two preimages of the $7$-sun are very useful for encoding binary value. 
Larger $k$-suns have more than two preimages, 
but it is possible to constrain larger $k$-suns 
to have exactly two preimages by embedding $7$-suns into the $k$-sun.

\begin{definition}
	Let $k \geq 9$ be an integer.
	Define the \emph{binary-enforced $k$-sun}, ${S}_k^{b}$, to be the graph built as follows:
	
	\begin{cem}
		\item Start with a $k$-sun, $S_k$.
		\item Label the vertices in the $k$-cycle as $v_1,\dots,v_k$, with subscripts taken modulo $k$.
		\item For each $i \in \{1,\dots,k\}$, 
				add a sequence of three triangles between 
				vertices $v_i$ and $v_{i+4}$, forming a $7$-sun (see Figure~\ref{subfig:12cyclegluing} for the case when $k=12$).

	\end{cem}
\end{definition}

\begin{lemma}\label{lma:12sun}
	Let $k \geq 9$ be an integer.
	A binary-enforced $k$-sun has exactly two preimages. 
	In the first preimage, the $k$-sun has a $k$-wheel preimage 
		and all 7-suns have 7-wheel preimages
		with a common center vertex. 
	In the second preimage, the $k$-sun has a squared cycle preimage 
		and all 7-suns have squared cycle preimages.
\end{lemma}

\begin{proof}
	Suppose $S_k^b$ is a triangular line graph with some preimage.
	Note that each 7-sun is a triangle-induced subgraph of $S_b^k$ 
		so by Lemmas \ref{lma:TriangleInducedClosure} 
			and \ref{lma:7sun}, it preimage must be a 7-wheel or a squared cycle.
	Each sequence of four consecutive triangles in the $k$-sun
		forms either a $4$-fan or $4$-strip in the preimage,
		depending on whether the preimage of the enclosing 7-sun is a wheel or squared cycle, respectively.
		
	Consider a single $7$-sun, $S$.
	If $S$ has a $7$-wheel preimage, then 
		the four triangles on the $k$-sun 
		form a $4$-fan, with the
		vertices on the $k$-cycle
		corresponding to edges 
		incident to the center of the fan.
	An adjacent $7$-sun $S'$ shares
		three consecutive triangles with $S$.
	The last triangle in the $k$-cycle and $S'$ (but not in $S$)
		has a vertex whose corresponding edge in the preimage
		is
		incident to the center of the 4-fan in $T^{-1}(S)$.
	Hence, the preimage of this triangle and the three 
		triangles in both $S$ and $S'$ create a 4-fan 
		and forces $S'$ to have a wheel preimage.
	By iterating through all consecutive wheels, 
		all $7$-suns in $S_b^k$ have wheel preimages.
	Hence, the 7-suns have either all wheel preimages
		or all squared cycle preimages.
	
	If all 7-suns have wheel preimages, 
		then each set of four consecutive triangles have
		$4$-fan preimages.
	Two consecutive 4-fans in the preimage 
		share the center with the common 3-fan.
	By iterating through consecutive $4$-fans around the
		$k$-cycle, all 7-wheel preimages share a common center.
	Hence, the $k$-sun preimage must be a wheel.
	
	If all 7-suns have squared cycle preimages, 
		then the consecutive $4$-strip preimages
		intersect at a $3$-fan.
	The edges of the $3$-fan which intersect the 
		other triangles of the $4$-strips are not
		incident to the center of the $3$-fan.
	Hence, every sequence of five triangles in the $k$-sun
		have triangle-strip preimages.
	Iterating on all lengths $\ell \in \{5,6,\dots,k-1\}$
		shows the $k$-sun must have 
		a squared-cycle preimage, since all sequences of $\ell$ 
		triangles have triangle-strip preimages.
	
	It remains to show that these two possibilities 
		are realizable with actual preimages.
	The all-wheels case can be viewed as a $k$-wheel, $W_k$,
		placed in the plane, with 3-fans glued
		between the ends of 4-fans in the wheel, sharing the center vertex.
	See Figure \ref{subfig:12wheelgluing} for an example of this process.

	\begin{figure}[ht]\centering
		\mbox{
			\subfigure[\label{subfig:12wheelgluing}
					$W_{12}$ with a gluing.]{
				\includegraphics[height=1.25in]{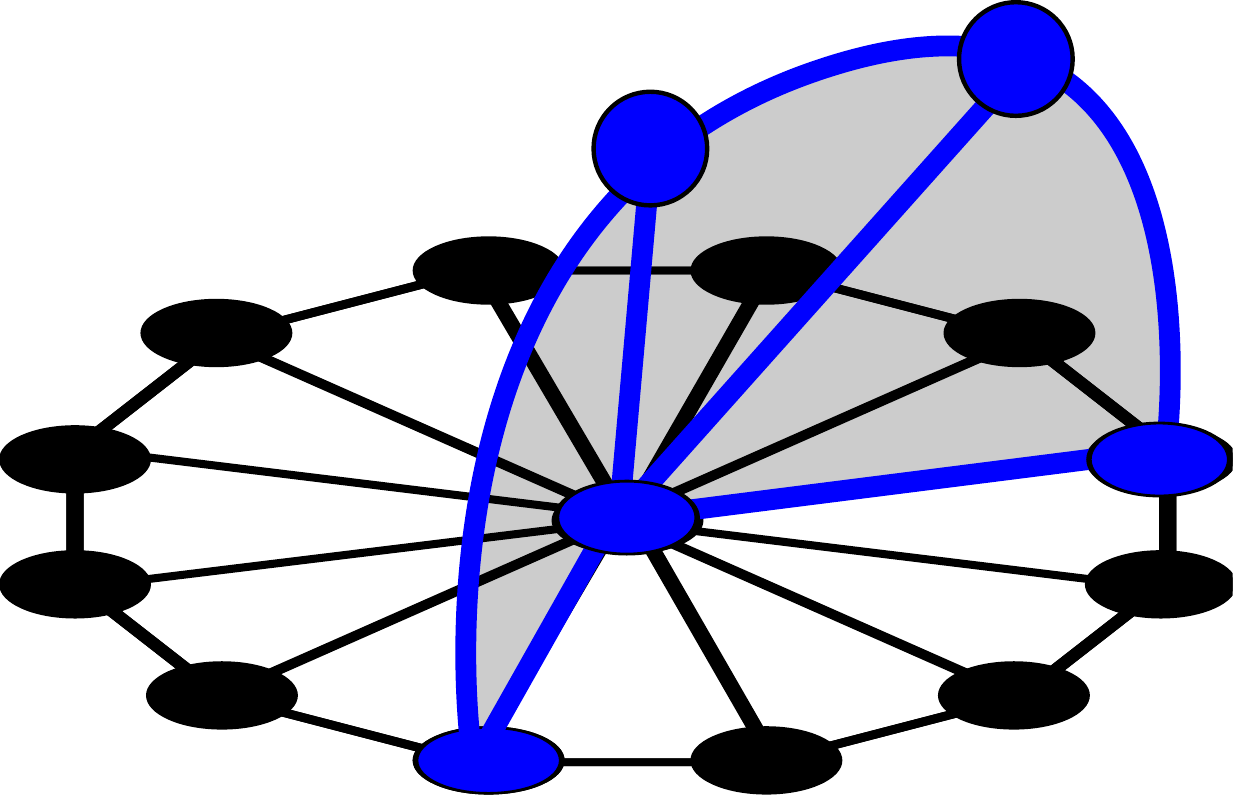}
			}
			\quad
			\subfigure[\label{subfig:12cyclecylinder}
					$C_{12}^2$ and $T(C_{12}^2)$ embedded on a cylinder.]{
				\includegraphics[height=1.25in]{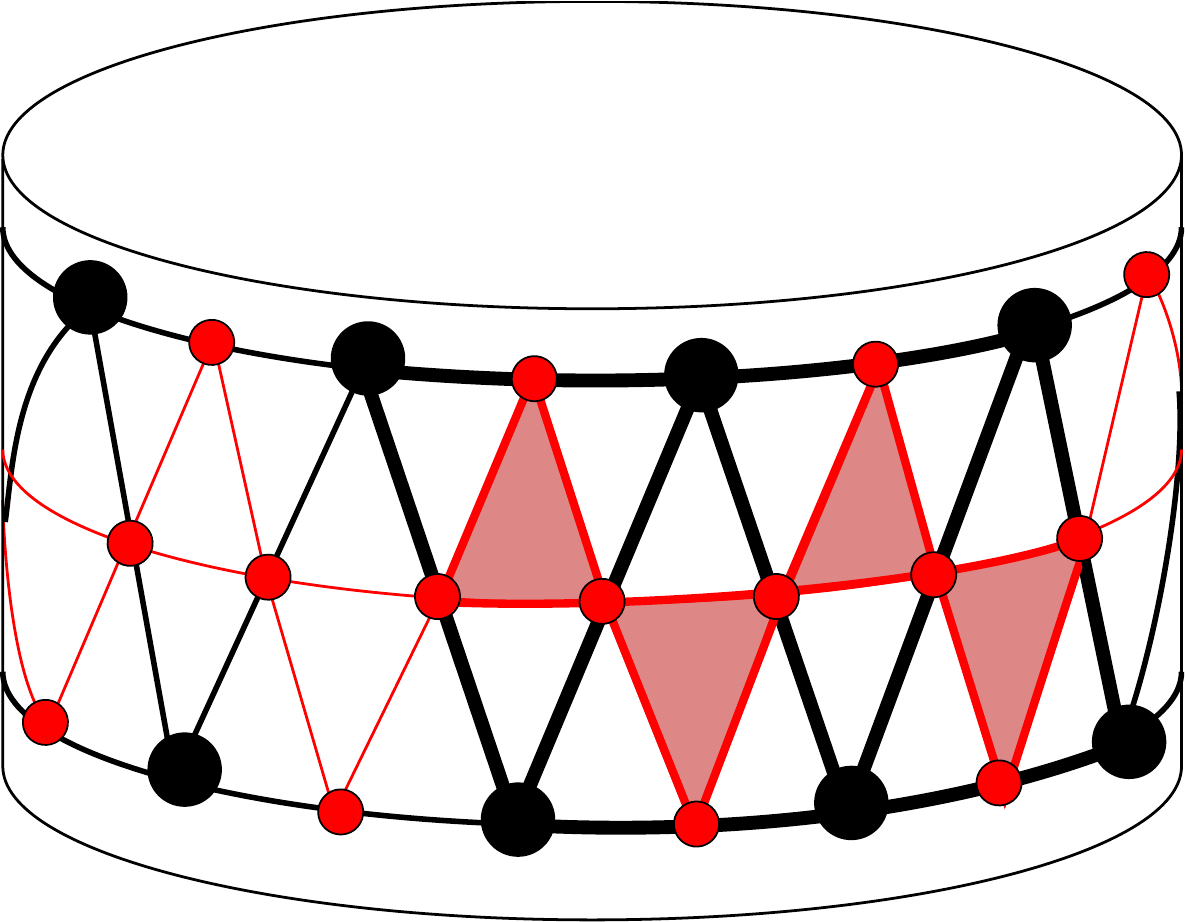}
			}
			\quad
			\subfigure[\label{subfig:12cyclegluing}
					$C_{12}^2$ and $T(C_{12}^2)$ with twisted gluing.]{
				\includegraphics[height=1.25in]{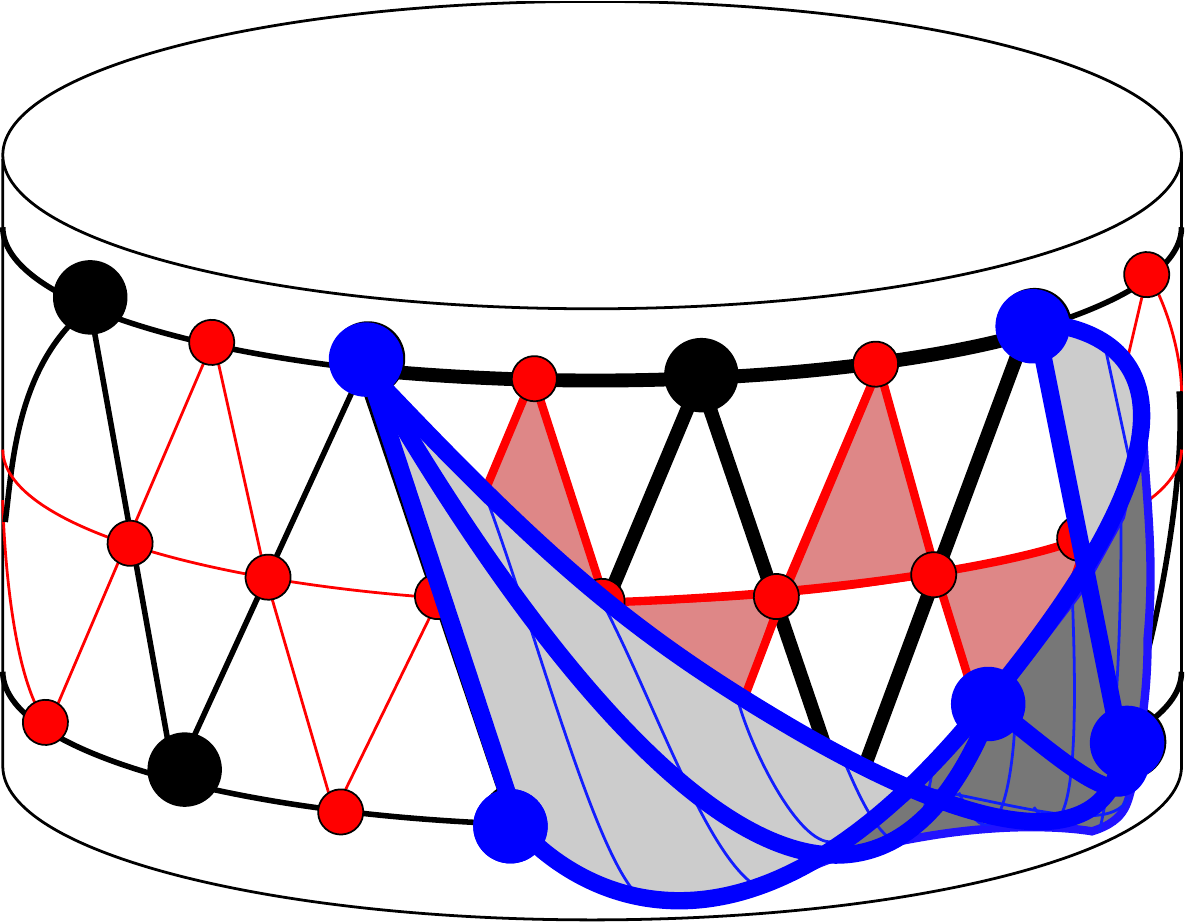}
			}
		}
		\caption{\label{fig:12cycle}Visualizing $S_{12}^b$ and its possible preimages}
	\end{figure}
		
	The all-cycles case can be viewed as the squared cycle, $C_k^2$,
		embedded on either the cylinder or M\"obius strip 
		(depending on the parity of $k$, see Observation \ref{obs:squaredcycleembed}).
	Then, a $3$-fan can be glued to the 
		last edges of each 4-triangle strip, 
		adding a twist to make a M\"obius strip.
	Figures \ref{subfig:12cyclecylinder} and \ref{subfig:12cyclegluing} 
		demonstrate this gluing procedure for $k = 12$.
\end{proof}
			
Armed with these structures, 
	we proceed to define gadgets 
	that encode the logic
	of a boolean formula.


\section{Encoding a Boolean Formula}

Let $\phi = \wedge_{j=1}^{m} C_j$ 
	be a formula 
	in conjunctive normal form 
	with $m$ clauses $C_1,\dots,C_m$
	on the variables $x_1,\dots, x_n$
	where each clause has size three. 
Label each clause as $C_{j} = u_{i_{j,1}} \vee u_{i_{j,2}} \vee u_{i_{j,3}}$, where
	each $u_{i_{j,k}}$ is a variable $x_{i_{j,k}}$ or its complement $\overline{x}_{i_{j,k}}$.
For our reduction, 
	we construct graph gadgets 
	that form a graph $G_\phi$ 
	which is a triangular line graph 
	if and only if $\phi$ is a satisfiable formula. 

We begin by defining a gadget which stores a binary value.
	
\subsection{Variable Gadgets}\label{subsection:variablegadget}

\def\EQUAL{{\sc{Equal}} }
\def\NOT{{\sc{Not}} }
\def\ROOT{{\sc{Root}} }

Begin by assigning variables to distinct copies of a {7-sun}.
The two preimages of a $7$-sun 
	correspond to 
	the two possible values 
	of a variable.
If the preimage is a 7-wheel, $W_7$, 
	then  that variable is assigned a false value, 
	and if it is a squared cycle, $C_7^2$, 
	then that variable is assigned a true value.

\begin{definition}[Variable Gadget]
For each variable $x_i$, create a $7$-sun labeled $H_{x_i}$.
Associate the value of $x_i$ with the preimage of $H_{x_i}$ as
	$x_i = 1$ if and only if $H_{x_i}$ has a $C_7^2$ preimage.
\end{definition}

There are two different ways to connect two 7-suns. 	
One connection guarantees the two preimages are isomorphic while
the other connection ensures the preimages are not isomorphic.
The connections act as logic gates in the graph.


\begin{definition}[\EQUAL Gadget]
		Let $S$ and $S'$ be two 7-suns and let
	$B$ and $B'$ be bowties in $S$ and $S'$, respectively.
	Identify the vertices of the bowties $B$ and $B'$ in the following way:
	
	\begin{cem}
		\item Identify the centers of the bowties $B$ and $B'$.
		\item In one triangle of $B$ and in one triangle of $B'$, identify the vertices of different
			degrees in $S$ and $S'$. 
		(Here, the vertex on the 7-cycle in $S$ is identified with 
			the vertex not on the 7-cycle in $S'$.)
		\item Repeat (2) for the remaining triangles in $B$ and $B'.$		
	\end{cem}	
\end{definition}

Figure \ref{fig:EQUALgadget} shows how the bowties are identified between two variable gadgets 
	joined by an \EQUAL gadget
	by highlighting which vertices are associated with the inner cycle
	of each 7-sun.
	
\begin{figure}[h]
\centering
	\mbox{
		\subfigure[\label{fig:EQUALgadget}The \EQUAL gadget.]{
			\includegraphics[width=1.25in]{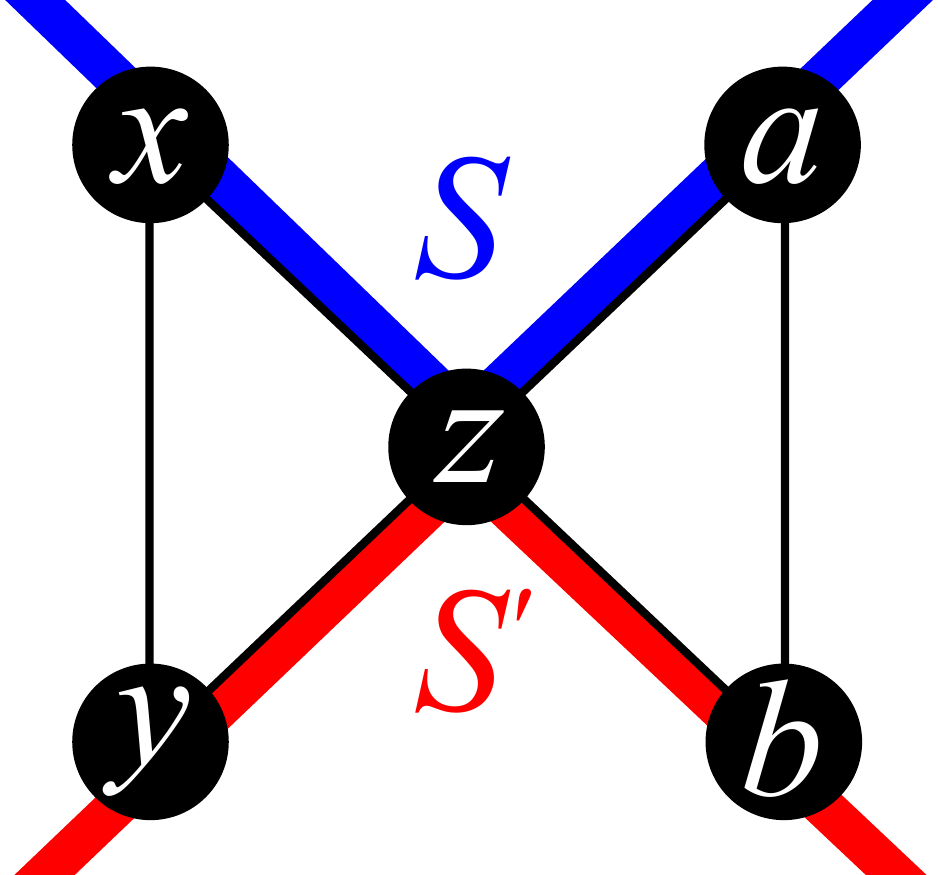}
		}
		\qquad
		\subfigure[\label{fig:NOTgadget}The \NOT gadget.]{
			\includegraphics[width=1.25in]{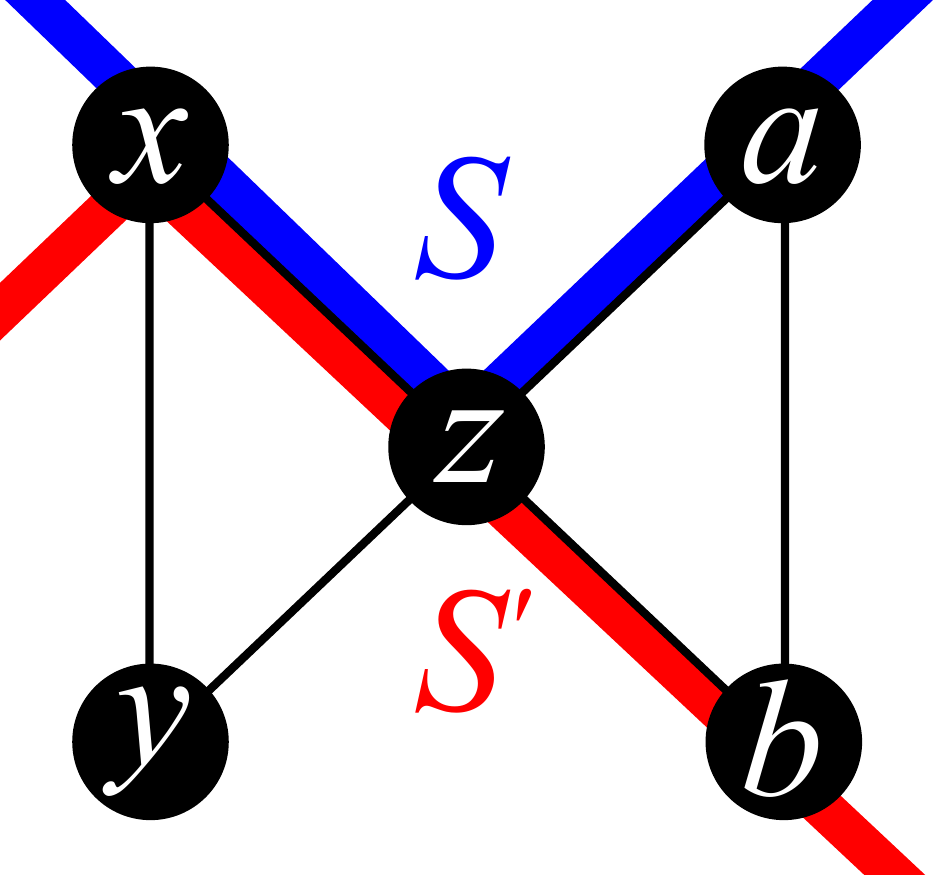}
		}
	}
\caption{\label{fig:EQUALandNOTgadgets}The \EQUAL and \NOT gadgets.}
\end{figure}

\begin{definition}[\NOT Gadget]
		Let $S$ and $S'$ be two 7-suns and let $B$ and $B'$ be bowties in $S$ and $S'$, respectively.
	Identify the vertices of the bowties $B$ and $B'$ in the following way:
	
	\begin{cem}
		\item Identify the centers of the bowties.
		\item In one triangle of $B$ and in one triangle of $B'$, identify the vertices of the same
			degree in $S$ and $S'$.
			 (Here, the vertices lying on the $7$-cycles of 
				$S$ and $S'$ are identified.)
		\item In the remaining triangles of $B$ and $B'$, identify the vertices
			of different degrees in $S$ and $S'$.  
			(Here, the vertex on the $k$-cycle in $S$
				are identified with the vertex not on the $k$-cycle in $S'$.)
	\end{cem}	
\end{definition}

Figure \ref{fig:NOTgadget} shows how the bowties are identified between two variable gadgets 
	joined by a \NOT gadget
	by highlighting which vertices are associated with the inner cycle
	of each variable gadget.

\begin{lemma}\label{lma:EQUALandNOTgadget}
	Consider two $7$-suns $S$ and $S'$ that intersect at a bowtie.
	
	\begin{enumerate}
		\item[1.] If the intersection is an \EQUAL gadget, then
				$T^{-1}(S)$ is isomorphic to $T^{-1}(S')$.
		\item[2.] If the intersection is a \NOT gadget, then
				$T^{-1}(S)$ is not isomorphic to $T^{-1}(S')$.
	\end{enumerate}
\end{lemma}

\begin{proof} Note that $S$ and $S'$ are triangle-induced subgraphs of the combined graph. 
	By Lemmas \ref{lma:TriangleInducedClosure} and \ref{lma:7sun}, the preimage of each is either a $W_7$ or a $C_7^2.$ 

	Let $x,y,z,a,b$ be the vertices in the bowtie of the gadget between
		$S$ and $S'$ as given in Figure \ref{fig:EQUALandNOTgadgets} and let $e_x, e_y, e_z, e_a$ and $e_b$, respectively, be the corresponding edges in $G.$
	Recall that $T^{-1}(S)$ and $T^{-1}(S')$ are isomorphic to either
		the 7-wheel $W_7$ or the squared cycle $C_7^2$.
	
	\begin{enumerate}
	\item[Case 1:] $S$ and $S'$ intersect at an \EQUAL gadget.
	
	Consider the case where $T^{-1}(S)$ is a wheel.
	Then, the edges  $e_x$, $e_z$, and $e_a$ are spokes of $T^{-1}(S)$ while 
	$e_y$ and $e_b$
		are consecutive edges on the $7$-cycle of $T^{-1}(S)$, 
		both of which are incident to $e_z$ at the same vertex.
	Hence, the edges $e_y$ and $e_b$ are incident. 
	Since $y$ and $b$ each have degree 4 in $S'$ 
		and $e_y$ and $e_b$ are incident in $G$ 
		(and therefore also in $T^{-1}(S')$), 
		it follows that $T^{-1}(S')$ is not a squared cycle. 
	Thus, $T^{-1}(S')$ is the wheel $W_7.$

	On the other hand, if $T^{-1}(S)$ is a squared cycle,
		then the edges $e_x$ and $e_a$ are not incident in $G.$
	Moreover, the edges $e_y$ and $e_b$ are not incident in $G.$
	Hence, $T^{-1}(S')$ is not a wheel and must be the squared cycle $C_7^2.$

	\item[Case 2:] $S$ and $S'$ intersect at a \NOT gadget.
	
	Consider the case where $T^{-1}(S)$ is a wheel.
	Then the edges $e_x$, $e_z$, and $e_a$ are spokes and $e_b$
		is on the $7$-cycle of $T^{-1}(S).$
	Hence, the edges for $e_x$ and $e_b$ are not incident in $G$ (and therefore also in $T^{-1}(S')$).
	It follows that $e_x$ and $e_b$ do not form spokes in $T^{-1}(S')$ and so $T^{-1}(S')$ is 
		not a wheel. Hence, $T^{-1}(S')$ is the squared cycle $C_7^2.$
		
	Now assume that $T^{-1}(S)$ is a squared cycle $C_7^2$.
	Then edges $e_x$, $e_y$, and $e_z$ form a triangle, 
	$e_a$ is incident to $e_z$ but not $e_x$, and
	$e_b$ is incident to $e_x$ in $G.$
 	This means that $e_x$, $e_b$, and $e_z$ share a common vertex in $G.$
	Moreover, $e_x$, $e_b$ and $e_z$ are incident in $G$ while $e_y$ and $e_b$ are not.
		Thus, $T^{-1}(S')$ is not the squared cycle $C_7^2$
		(since if $T^{-1}(S')$ is a squared cycle, then $e_b$ and $e_y$ would be incident
			while $e_b$ and $e_x$ would not be). Consequently, $T^{-1}(S')$ is isomorphic to the $7$-wheel $W_7.$ \qedhere
	\end{enumerate}
\end{proof}

In order to provide multiple connection points for clauses to access the value of $x_i$ or $\overline{x}_i$,
	we connect a sequence of $7$-suns to each $7$-sun $H_{x_i}.$ 
To that end, we assign three bowties within each $7$-sun two of which are for special use with the \EQUAL and \NOT gadgets.
One bowtie is labeled \ROOT 
		which is used to connect 
		to other variable gadgets
		via \EQUAL or \NOT gadgets, 
	one is labeled \EQUAL 
		which is reserved 
		for other variable gadgets 
		to connect 
		via an \EQUAL gadget, 
	and one is labeled \NOT 
		which is reserved 
		for other variable gadgets 
		to connect 
		via a \NOT gadget.
Figure \ref{fig:8sunattach} shows the three attachment points
	and a symbolic representation of the connections.

\begin{figure}[h]\centering
	\mbox{
		\subfigure[The attachment points]{
			\includegraphics[height=1in,width=1in]{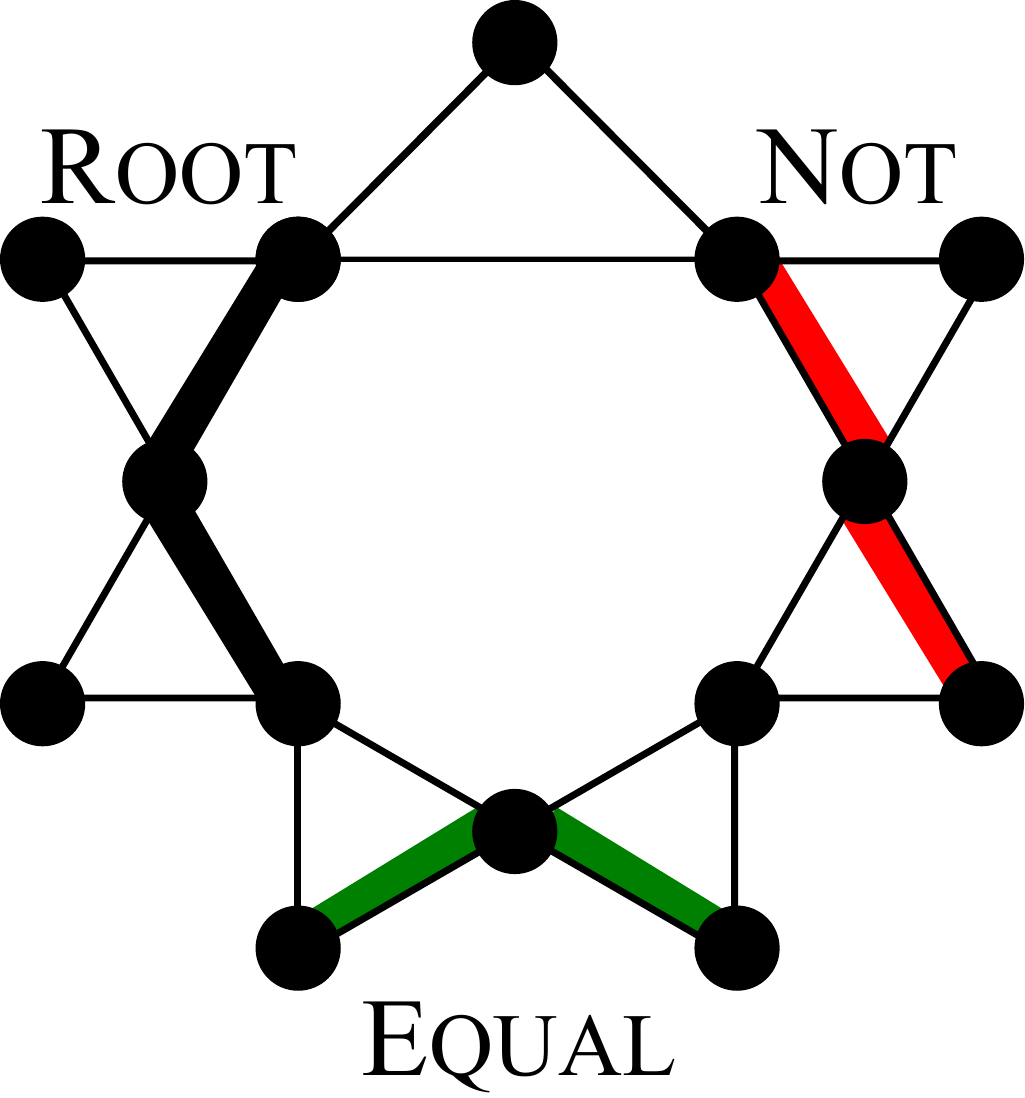}
		}
		\hskip 1truecm
		\subfigure[\label{7SunAttachMultiple}How they attach]{
			\includegraphics[height=1in]{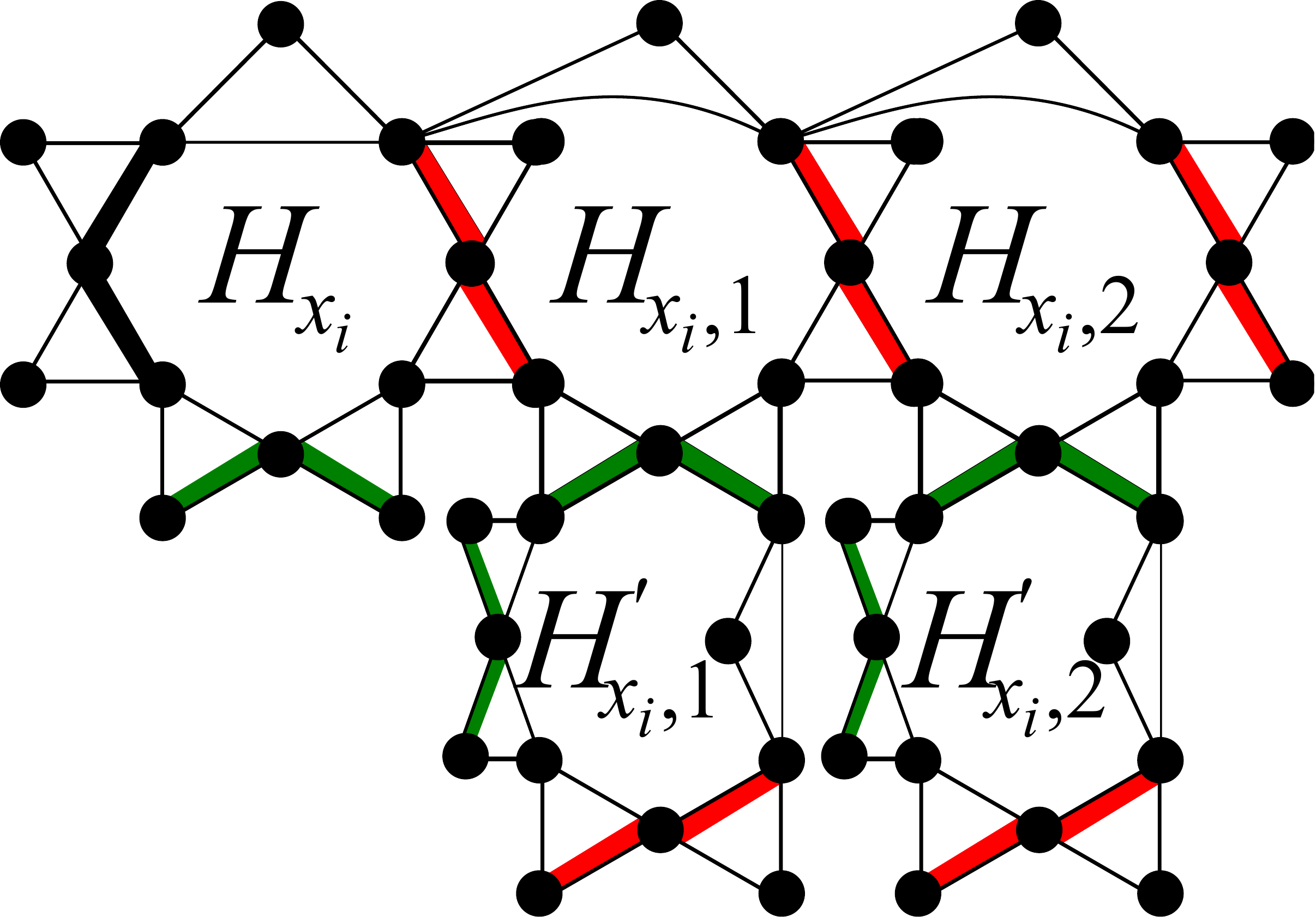}
		}
		\hskip 1 truecm
		\subfigure[\label{7SunAttachSymbolic}A symbolic representation]{
			\includegraphics[height=1in]{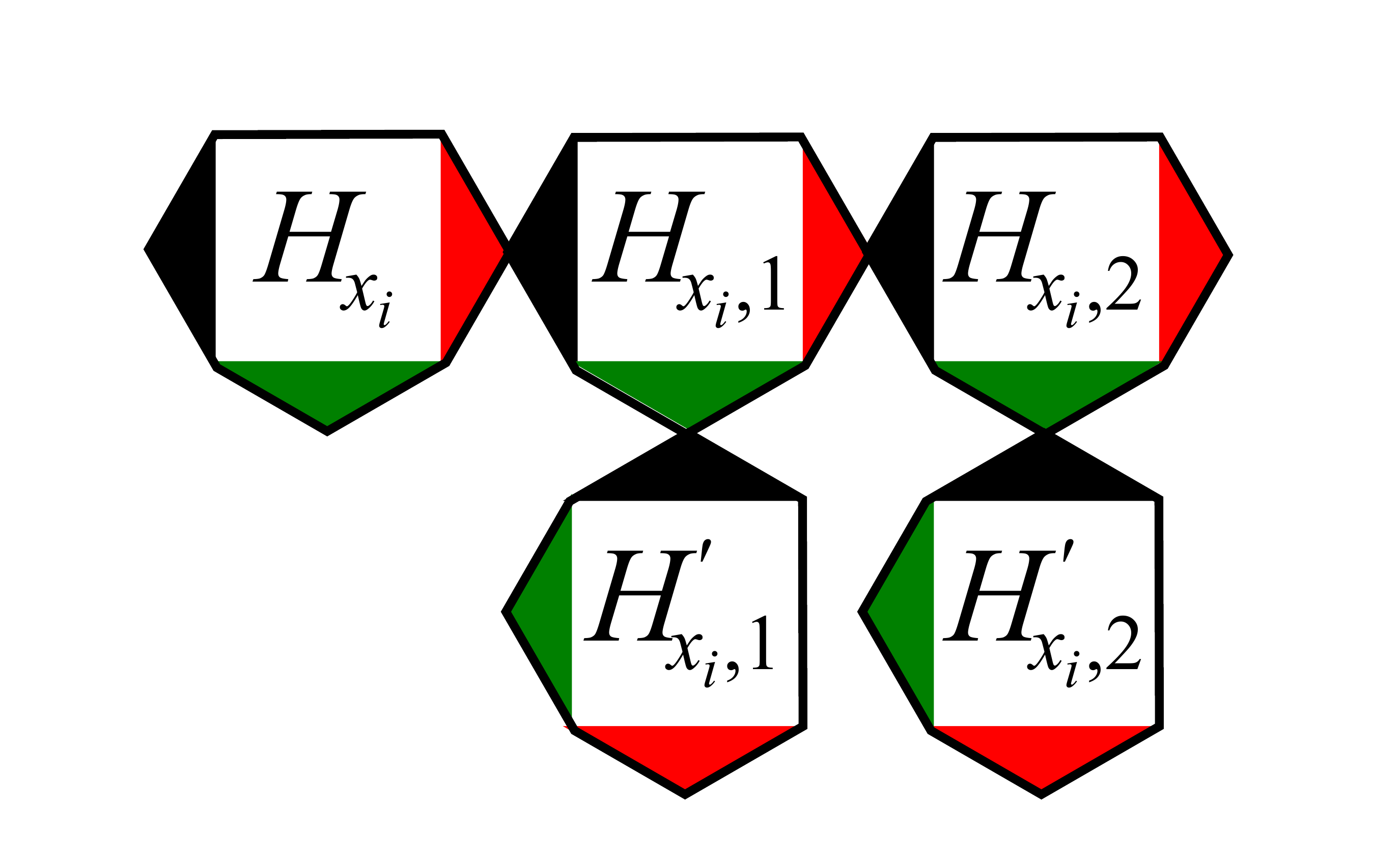}
		}
	}
	\caption{\label{fig:8sunattach}The \ROOT, \NOT, and \EQUAL attachment points of a 7-sun 	
			and how they attach.}
\end{figure}


\begin{definition}[Wire]
	A \emph{wire} is a sequence
		$H_0,H_1,\dots, H_k$ of 7-suns 
		such that for $i\in\{1,2, \ldots, k\}$,  
		$H_i$ is attached 
		at the \ROOT to $H_{i-1}$ 
		via a \NOT gadget.
\end{definition}

Observe that the preimage of a wire 
		is a set of overlapping $W_7$s and $C_7^2$s
		so that $T^{-1}(H_i) \cong T^{-1}(H_0)$
		if and only if $i$ is even.
Using a wire as a central line storing the value of a variable $x_i$
	and its complement $\overline{x}_i$,
	we build a larger structure ({\it variable cluster}) that allows connections to other variables
	via clause gadgets.

\subsection{Variable Clusters}

For each variable $x_i$, form a wire of length $2m+1$: $H_{x_i},H_{x_i,1},\dots,H_{x_i,2m}$.
For each $j \in \{1, 2, \ldots, 2m\}$, 
	attach a $7$-sun $H_{x_i,j}'$ via an \EQUAL gadget to the 
	7-sun $H_{x_i,j}$ in the wire 
	(see Figures~\ref{7SunAttachMultiple} and \ref{7SunAttachSymbolic} 
		for an example when $m = 1$). 
Note that $T^{-1}(H'_{x_i,j}) \cong T^{-1}(H_{x_i,j})$ 
	for all $j$ with $1\le j\le 2m$. 
Hence, $T^{-1}(H'_{x_i,j}) \cong T^{-1}(H_{x_i})$ 
	if and only if $j$ is even. 
Consider the value stored by $T^{-1}(H'_{x_i,j})$ 
	to be $x_i$ if $j$ is even and $\overline{x_i}$ when $j$ is odd.

Our construction of variable clusters is almost complete. 
However, we need to make these variable clusters 
	interact with each other 
	via clause gadgets 
	that allow all the satisfying cases 
	in the truth table 
	of a clause. 
To allow all true cases, the clause requires larger suns than the $7$-sun,
	but requires the suns to have only the two canonical preimages.
Hence, we must incorporate the binary-enforced $k$-sun into
	our construction for large enough $k$.
It is sufficient to have $k = 12$, as we will see in Section \ref{sec:clauses}.

\begin{definition}[Large Variable Gadget]
	Given a variable $x_i$ and integer $j \in \{1,\dots, m\}$, the \emph{large variable gadget}
		$V_{x_i,j}$ is given as a binary-enforced $12$-sun, with one of its $7$-suns
		labeled $H_{x_i,j}'$.
\end{definition}

Lemma~\ref{lma:12sun} demonstrated that a large variable gadget has exactly two preimages. 
To complete the variable cluster for $x_i$, 
	take the $2m$ large variable gadgets $V_{x_i,1},\dots,V_{x_i,2m}$ 
	and identify each $H_{x_i,j}'$ with the 
	$7$-sun having the same subscript in the the wire gadget for $x_i.$
		
\begin{definition}[Variable Clusters]
	Each variable $x_i$ is represented by 
		a \emph{variable cluster} composed of
		a wire $H_{x_i}, H_{x_i,1},\dots, H_{x_i,2m}$ of length $2m+1$,
		variable clusters $V_{x_i,1},\dots,V_{x_i,2m}$
		with the $7$-suns $H_{x_i,j}'$ attached via \EQUAL gadgets
		to $H_{x_i,j}$ in the wire.
\end{definition}

\begin{lemma}\label{lma:variablecluster}
	A variable cluster has exactly two preimages:
	
		\begin{enumerate}
		
		\item[(a)] $H_{x_i}$ has a wheel preimage while 
			$H_{x_i,j}$ has a wheel preimage when $j$ is even 
			and a squared cycle preimage when 
			$j$ is odd.
			The variable gadget $V_{x_i,j}$ has a preimage
			that contains a $12$-wheel if and only if $j$ even.  
	
		\item[(b)] $H_{x_i}$ has a squared cycle preimage while 
			$H_{x_i,j}$ is a squared cycle preimage when $j$ is even 
			and a wheel preimage when $j$ is odd.
			The variable gadget $V_{x_i,j}$ has a preimage that contains a 
				$12$-wheel if and only if $j$ is odd.
			
	\end{enumerate}
\end{lemma}

\begin{proof}
	Since each $H_{x_i}$, $H_{x_i,j}$, or $H_{x_i,j}'$ 
		is a triangle-induced subgraph isomorphic to a $7$-sun,  
		its preimage is either a wheel or a 
		squared cycle by Lemmas \ref{lma:TriangleInducedClosure} and \ref{lma:7sun}.
	By the properties of the \NOT gadget (Lemma \ref{lma:EQUALandNOTgadget}), 
		the preimages of $H_{x_i,j}$ in the wire alternate
		between a wheel and a squared cycle
		depending on the preimage of $H_{x_i}$.
	Since $H_{x_i,j}'$ is connected to $H_{x_i,j}$ 
		via an \EQUAL gadget, 
		the preimage of $H_{x_i,j}'$ is isomorphic 
		to the preimage of $H_{x_i,j}.$ 
	Moreover, the preimage of $H_{x_i,j}'$  
		determines the preimage of $V_{x_i,j}$ by Lemma \ref{lma:12sun}.
\end{proof}

Figure \ref{fig:variablecluster} shows a diagram 
	of the variable cluster for $x_i$ 
	and the associated variable value correspondence.

\begin{figure}[ht]\centering
	\includegraphics[height=1.75in]{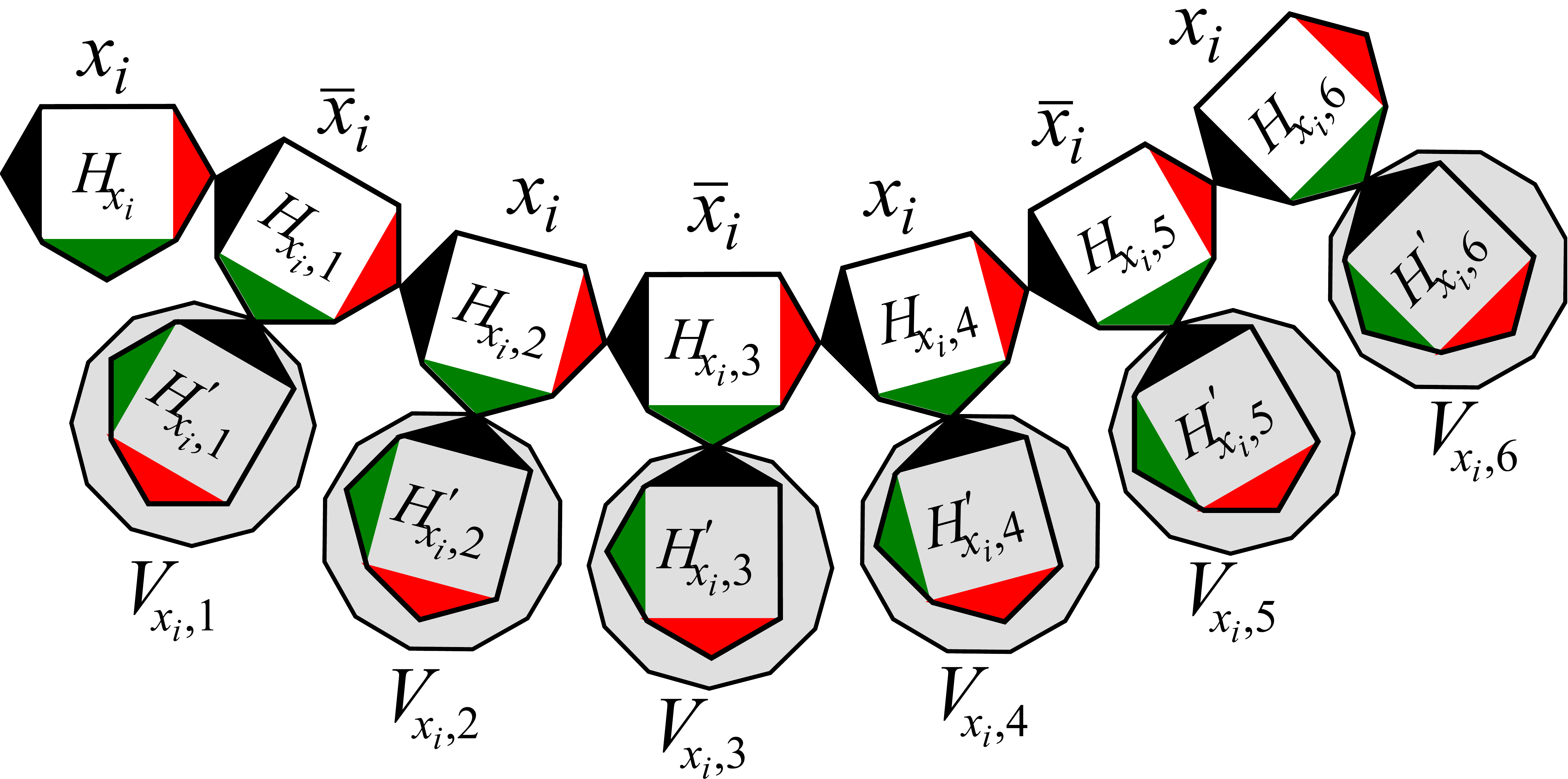}
	\caption{\label{fig:variablecluster}The variable cluster for $x_i$ with $m=3$.}
\end{figure}


\section{Enforcing Clause Satisfaction}\label{sec:clauses}

It remains to show how to encode each clause $C_j$ of $\phi$ into a graph $H$
	so that $H$ is a triangular line graph if and only if $\phi$ is satisfiable.
The method combines large variable gadgets using clause gadgets so 	
	that the large variable gadgets corresponding to the three literals in a 
	clause cannot all have wheels as preimages.
This prevents a preimage of $H$ from existing unless 
	the variable gadget preimages correspond to
	an assignment of the variables that satisfies $\phi$.

Consider three 12-suns $G_1$, $G_2$, and $G_3$. 
For $\ell\in \{1, 2, 3\}$, 
	let $\langle a_{\ell, 1}, a_{\ell, 2}, a_{\ell, 3}\rangle$ 
	and $\langle b_{\ell, 1},b_{\ell,2},b_{\ell,3}\rangle$ 
	be two triangles of maximum distance apart 
	in $G_{\ell}$ 
	where $a_{\ell,3}$ and $b_{\ell,3}$ are vertices of degree $2$ in $G_{\ell}.$ 
It follows that that are five triangles between 
	$\langle a_{\ell, 1}, a_{\ell, 2}, a_{\ell, 3}\rangle$ 
	and $\langle b_{\ell, 1},b_{\ell,2},b_{\ell,3}\rangle$ 
	in $G_{\ell}$ 
	for each $\ell\in\{1,2,3\}$ 
	and $a_{\ell, 1}, a_{\ell, 2}, b_{\ell,1}$ and $b_{\ell,2}$ 
	all lie on the $12$-cycle of $G_{\ell}.$ 
Assume that the vertices 
	on the $12$-cycle of $G_{\ell}$ 
	appear in the order $a_{\ell, 1}, a_{\ell, 2}, b_{\ell,1}, b_{\ell,2.}$		
Identify triangles by identifying the vertices
		(a) $a_{\ell,1}$ with $b_{\ell+1,1}$, 
		(b) $a_{\ell,2}$ with $b_{\ell+1,3}$, and 
		(c) $a_{\ell,3}$ with $b_{\ell+1,2}$ for $\ell \in \{1, 2, 3\}$ 
		(with subscripts taken modulo $3).$ 
This construction is shown in Figure \ref{fig:clausegadget}. 

\begin{figure}[ht]\centering
	\includegraphics[height=1.25in]{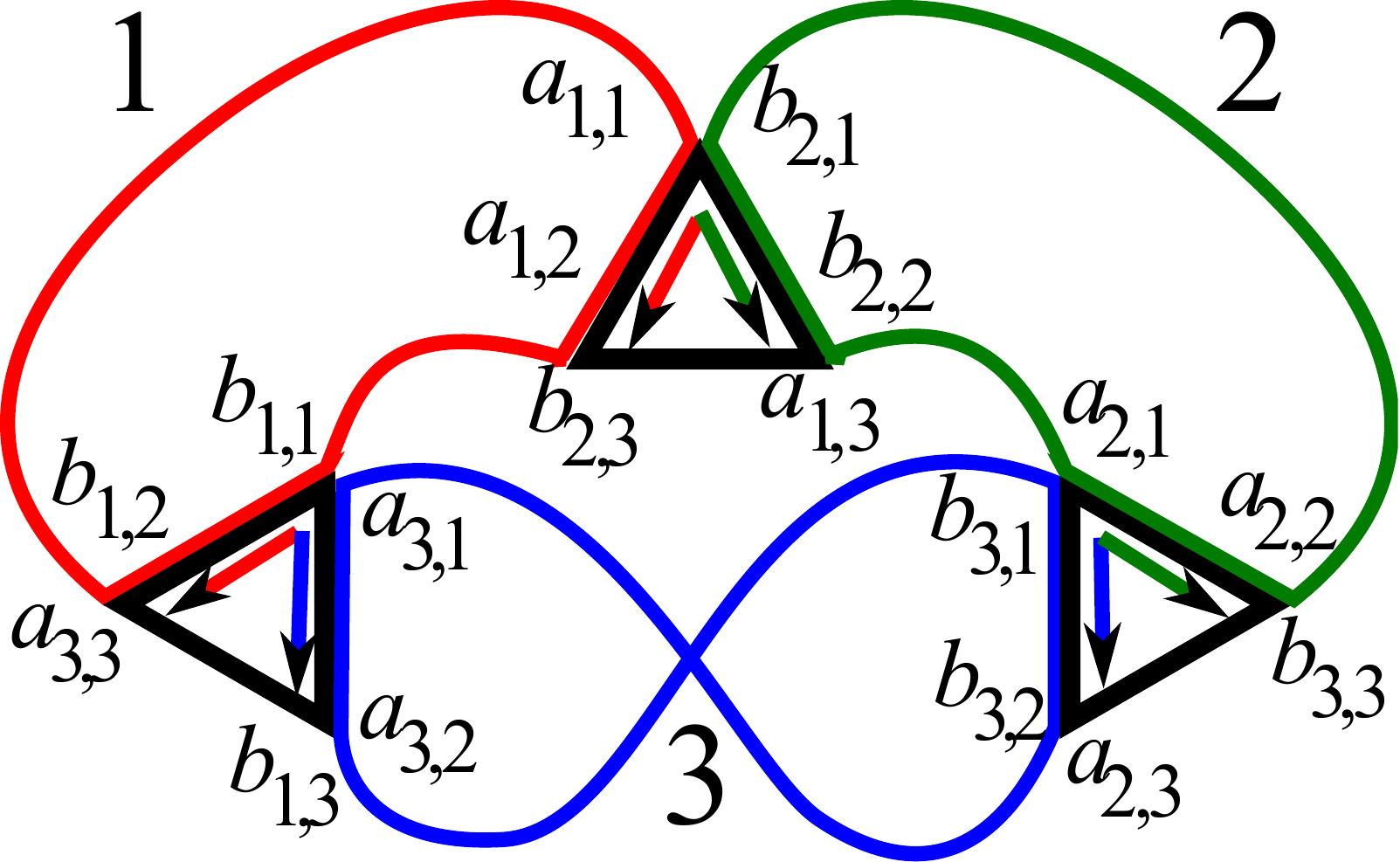}
	\caption{\label{fig:clausegadget}Twisting three 12-suns together}
\end{figure}

\begin{definition}[Clause Gadget]
	Let $C_j = u_{i_{j,1}} \vee u_{i_{j,2}} \vee u_{i_{j,3}}$ 
		be the $j$th clause in $\phi$.	
	If $u_{i_{j,\ell}}=x_{i_{j,\ell}}$, 
		set $k_\ell = 2j$,
		otherwise $u_{i_{j,\ell}} = \overline{x}_{i_{j,\ell}}$
		and set $k_\ell = 2j-1, (\ell\in\{1, 2, 3\}$.
	The \emph{clause gadget} for $C_j$ 
		joins the 12-suns within the three
		large variable gadgets $V_{x_{i_{j,1}},k_{1}}, V_{x_{i_{j,2}},k_{2}}, V_{x_{i_{j,3}},k_{3}}$
		using the construction above.
\end{definition}

The important property of this clause gadget is that 
	the possible preimages correspond to the satisfying assignments
	of a disjunction of three literals.
	
\begin{lemma}\label{lma:clause}
	Let $G_1, G_2,$ and $G_3$ be the 12-suns used in a clause gadget. If the clause gadget has a preimage, 
		then at least one of $G_1, G_2$ and $G_3$ has $C_{12}^2$ as a preimage.
\end{lemma}

\begin{proof}
	The existence of preimages to the clause gadget
		with zero, one, or two wheels is shown by the explicit 
		construction found in Appendix \ref{apx:ClauseGadget}
		and Tables \ref{tbl:ClauseGadget}-\ref{tbl:2wheels}. 
	It remains to show that the clause gadget does not have a preimage where
		each $G_\ell$ has a wheel preimage.
	We will show that this is not a triangular line graph,
				as more adjacencies are required.
	
	Suppose for the sake of contradiction that 
		there exists a preimage of the clause gadget
		where each 12-sun has a wheel preimage.
	Since the preimage of the induced $12$-sun $G_{\ell}$ is a wheel, 
		the vertices $a_{\ell,1}$, $a_{\ell,2}$, $b_{\ell,1}$, and $b_{\ell,2}$ 
		correspond to spokes in $T^{-1}(G_{\ell})$ for $\ell\in\{1, 2, 3\}.$
	These spokes share a common endpoint which 
		is the center of the wheel $T^{-1}(G_\ell)$.
	Since $a_{1,1} = b_{2,1}$, 
		 $a_{2,1} = b_{3,1}$, and 
		$a_{3,1} = b_{1,1}$,						
			it follows that the centers of the wheels 
			$T^{-1}(G_1)$, $T^{-1}(G_2)$, and $T^{-1}(G_3)$ 
			induce a triangle in the preimage.
	However, $a_{1,1}$ is not adjacent to either 
		$a_{2,1}$ nor $a_{3,1}$ in the clause gadget, forming a contradiction.
\end{proof}

With this clause gadget, we prove the main theorem.

\begin{thm}\label{thm:NPC}
	Recognizing if a graph is a triangular line graph is NP-complete.
\end{thm}

\begin{proof}
	Given a graph $H$ on $n$ vertices $v_1,\dots,v_n$,
		non-deterministically select pairs $e_1,\dots,e_n$
		from the set $\{1,\dots, 2n\}$ 
		forming a graph $G = (\{1,\dots,2n\}, \{e_1,\dots,e_n\})$.
	Then, check that the adjacencies of $e_ie_j$ in $T(G)$
		match those of $v_iv_j$ in $H$, for each $i, j \in \{1,\dots, n\}$.
	This takes polynomial time, so the problem is in NP.
	
	To show hardness, we reduce from 3-SAT by converting
		a 3-CNF formula $\phi$ on $n$ variables and $m$ clauses
		into the graph $G_\phi$:
		
	\begin{cem}	
		\item For each variable $x_i$, $i=1, 2, \ldots, n$, 
				form a variable cluster using a wire of length $2m+1.$
		\item Combine these  $n$ variable clusters via the clause gadgets for clauses $C_1, C_2, \dots, C_m$ to obtain graph $G_\phi.$ 
	\end{cem}
		
	This process can be done in polynomial time.
	Note that each gadget is a triangle-induced subgraph of $G_\phi$
		and so Lemma \ref{lma:TriangleInducedClosure}
		restricts the preimage of $G_\phi$ to 
		include only allowed preimages of each gadget.

\def\vx{{\mathbf x}}
	We now show that $\phi$ 
		is satisfiable 
		if and only if $G_\phi$ 
			is a triangular line graph.
	Given an assignment $\vx$, 
		consider a possible preimage 
		for $G_\phi$ 
		by setting the
		variable cluster 
		for $x_i$ 
		to have $T^{-1}(H_{x_i})=W_7$ 
		if $x_i = 0$
		and $T^{-1}(H_{x_i})\cong C_7^2$ 
		if $x_i = 1$.
	The rest of the preimage of the variable cluster for $x_i$ propagates according 
		to Lemma \ref{lma:variablecluster}.
	If $\vx$ satisfies $\phi$, then every clause $C_j$
		 is simultaneously satisfied.
	Hence, the clause gadget for $C_j$ 
		has at least one satisfying variable
		whose large variable gadget
		has preimage containing the squared cycle $C_{12}^2$.
	Thus, each clause gadget has a preimage as well.
		
	On the other hand, 
		if there is a triangular line graph preimage for $G_\phi$, 
		then each clause gadget has at least one large variable gadget
		whose preimage contains the squared cycle $C_{12}^2.$
	Hence, a satisfying assignment $\vx$ can be formed by setting
		each $x_i = 1$ if and only if $T^{-1}(H_{x_i}) \cong C_{7}^2$.
\end{proof}

This shows that the general problem is NP-complete, 
	but our construction 
	shows hardness for a special class of triangular line graphs.

\begin{cor}
	Let $\mathcal{G}$ be the class of graphs $G$ where each edge in $E(G)$ 
		is in exactly one triangle in $G$.
	Deciding if a graph $G \in \mathcal{G}$ is a triangular line graph
		is NP-complete.
\end{cor}

\begin{proof}
	The construction of Theorem \ref{thm:NPC} produced a graph $G_{\phi}$ 
		which is a triangular line graph if and only if $\phi$
		is a satisfiable 3-CNF formula.
	This $G_\phi$ has every edge in a unique triangle, so $G_\phi$ is in the class 
	$\mathcal{G}$, and the above reductions works for this problem.
\end{proof}

This corollary is interesting since if a preimage exists in this class,
	it is immediately clear which edges form triangles
	and which subgraphs are triangle-induced subgraphs.
In this class, preimages have no cliques of size four.


\section{Conclusion}

The hardness of recognizing triangular line graphs is perhaps surprising 
	given the polynomial-time algorithms which recognize line graphs.
This hardness
	 is consistent 
	with the complexity of Le's characterization, 
	which is unlikely to be substantially 
	simplified for triangular line graphs.
It is interesting to consider Gallai graphs, 
	which are constructed
	from line graphs and triangular line graphs 
	by $\Gamma(G) = L(G) - E(T(G))$.
Even though these graphs are closely related to triangular line graphs,
	the constructions and proofs in this work 
	do not immediately generalize to show 
	the hardness of recognizing these graphs.
The complexity for other generalizations, such as $k$-Gallai and $k$-anti-Gallai graphs,
	or $k$-in-$\ell$ graphs,
	remains unknown.


\bibliographystyle{alpha}
\bibliography{sources.bib}

\clearpage
\appendix

\section{Clause Gadgets and Preimages}\label{apx:ClauseGadget}

To construct the clause gadget explicitly, we must first describe three 12-suns, $S_1, S_2, S_3$.
Each $S_\ell$ has 24 vertices, $S_{\ell,i}$ for $i \in \{0,\dots,23\}$.
The first index of $S_{\ell,i}$ is modulo three while the second index is modulo 24.
The even $i$ vertices $S_{\ell,i}$ are in the 12-cycle and are adjacent to vertices $S_{\ell,j}$ where
	$j \in \{i-2, i-1, i+1, i+2\}$.
The odd $i$ vertices have degree two and are adjacent to vertices $S_{\ell,j}$ where 
	$j \in \{i-1,i+1\}$.
	
The triangles in $S_\ell$ used for the clause gadget are given by 
	\begin{align*}
		a_{\ell,1} &= S_{\ell,0}     &    b_{\ell,1} &= S_{\ell,12}\\
		a_{\ell,2} &= S_{\ell,2}     &    b_{\ell,2} &= S_{\ell,14}\\
		a_{\ell,3} &= S_{\ell,1}     &    b_{\ell,3} &= S_{\ell,13}
	\end{align*}
	
Then, by merging the $a$-triangle in $S_\ell$ to the $b$-triangle in $S_{\ell+1}$, 
	the following vertices are identified by equality:
	
	\begin{align*}
		S_{\ell,0} = a_{\ell,1} &= b_{\ell+1,1} = S_{\ell+1,12}\\
		S_{\ell,2} = a_{\ell,2} &= b_{\ell+1,3} = S_{\ell+1,13}\\
		S_{\ell,1} = a_{\ell,3} &= b_{\ell+1,2} = S_{\ell+1,14}\\
	\end{align*}	

The full adjacency list for the clause gadget given by these $S_1, S_2, S_3$ is given
	in Table \ref{tbl:ClauseGadget} on page \pageref{tbl:ClauseGadget}.
Preimages for the clause gadget are given in the following tables.
Each of these preimages labels the adjacencies with the corresponding vertex
	in the clause gadget that edge represents.
\begin{itemize}
	\item Table \ref{tbl:0wheels} on page \pageref{tbl:0wheels} shows the preimage with zero wheels and three cycles.
	\item Table \ref{tbl:1wheel} on page \pageref{tbl:1wheel} shows the preimage with one wheel and two cycles.
	\item Table \ref{tbl:2wheels} on page \pageref{tbl:2wheels} shows the preimage with two wheels and one cycle.
\end{itemize}	

\begin{table}[p]
\small
\centering
\mbox{

\renewcommand{\arraystretch}{1.25}
\begin{tabular}[h]{c|lllll}
	 Vertex & \multicolumn{5}{c}{Adjacencies}\\
\hline 
	  $S_{1,0\equiv2,12}$ & 
	 $S_{1,1\equiv2,14}$ &
	 $S_{1,23}$ &
	 $S_{2,10}$ &
	 $S_{2,11}$ &
\\ &
	 $S_{1,22}$ &
	 $S_{1,2\equiv2,13}$ &\\

\hline 
	  $S_{1,1\equiv2,14}$ & 
	 $S_{1,0\equiv2,12}$ &
	 $S_{1,2\equiv2,13}$ &
	 $S_{2,16}$ &
	 $S_{2,15}$ &
\\
\hline 
	  $S_{1,2\equiv2,13}$ & 
	 $S_{1,0\equiv2,12}$ &
	 $S_{1,1\equiv2,14}$ &
	 $S_{1,4}$ &
	 $S_{1,3}$ &
\\
\hline 
	  $S_{1,3}$ & 
	 $S_{1,2\equiv2,13}$ &
	 $S_{1,4}$ &
\\
\hline 
	  $S_{1,4}$ & 
	 $S_{1,2\equiv2,13}$ &
	 $S_{1,6}$ &
	 $S_{1,5}$ &
	 $S_{1,3}$ &
\\
\hline 
	  $S_{1,5}$ & 
	 $S_{1,6}$ &
	 $S_{1,4}$ &
\\
\hline 
	  $S_{1,6}$ & 
	 $S_{1,8}$ &
	 $S_{1,7}$ &
	 $S_{1,5}$ &
	 $S_{1,4}$ &
\\
\hline 
	  $S_{1,7}$ & 
	 $S_{1,8}$ &
	 $S_{1,6}$ &
\\
\hline 
	  $S_{1,8}$ & 
	 $S_{1,7}$ &
	 $S_{1,6}$ &
	 $S_{1,10}$ &
	 $S_{1,9}$ &
\\
\hline 
	  $S_{1,9}$ & 
	 $S_{1,8}$ &
	 $S_{1,10}$ &
\\

\hline 
	  $S_{1,10}$ & 
	 $S_{3,0\equiv1,12}$ &
	 $S_{1,8}$ &
	 $S_{1,11}$ &
	 $S_{1,9}$ &
\\
\hline 
	  $S_{1,11}$ & 
	 $S_{3,0\equiv1,12}$ &
	 $S_{1,10}$ &
\\
\hline 
	  $S_{1,15}$ & 
	 $S_{3,1\equiv1,14}$ &
	 $S_{1,16}$ &
\\
\hline 
	  $S_{1,16}$ & 
	 $S_{3,1\equiv1,14}$ &
	 $S_{1,18}$ &
	 $S_{1,17}$ &
	 $S_{1,15}$ &
\\
\hline 
	  $S_{1,17}$ & 
	 $S_{1,18}$ &
	 $S_{1,16}$ &
\\
\hline 
	  $S_{1,18}$ & 
	 $S_{1,20}$ &
	 $S_{1,17}$ &
	 $S_{1,16}$ &
	 $S_{1,19}$ &
\\
\hline 
	  $S_{1,19}$ & 
	 $S_{1,20}$ &
	 $S_{1,18}$ &
\\
\hline 
	  $S_{1,20}$ & 
	 $S_{1,19}$ &
	 $S_{1,21}$ &
	 $S_{1,18}$ &
	 $S_{1,22}$ &
\\
\hline 
	  $S_{1,21}$ & 
	 $S_{1,20}$ &
	 $S_{1,22}$ &
\\
\hline 
	  $S_{1,22}$ & 
	 $S_{1,0\equiv2,12}$ &
	 $S_{1,20}$ &
	 $S_{1,21}$ &
	 $S_{1,23}$ &
\\
\hline 
	  $S_{1,23}$ & 
	 $S_{1,0\equiv2,12}$ &
	 $S_{1,22}$ &
\\
\hline 
	  $S_{2,0\equiv3,12}$ & 
	 $S_{2,1\equiv3,14}$ &
	 $S_{2,22}$ &
	 $S_{2,23}$ &
	 $S_{3,10}$ &
\\ &
	 $S_{2,2\equiv3,13}$ &
	 $S_{3,11}$ &
\\

\hline 
	  $S_{2,1\equiv3,14}$ & 
	 $S_{2,0\equiv3,12}$ &
	 $S_{2,2\equiv3,13}$ &
	 $S_{3,15}$ &
	 $S_{3,16}$ &
\\

\hline 
	  $S_{2,2\equiv3,13}$ & 
	 $S_{2,0\equiv3,12}$ &
	 $S_{2,1\equiv3,14}$ &
	 $S_{2,3}$ &
	 $S_{2,4}$ &
\\
\hline 
	  $S_{2,3}$ & 
	 $S_{2,2\equiv3,13}$ &
	 $S_{2,4}$ &
\\
\hline 
	  $S_{2,4}$ & 
	 $S_{2,2\equiv3,13}$ &
	 $S_{2,3}$ &
	 $S_{2,5}$ &
	 $S_{2,6}$ &
\\
\hline 
	  $S_{2,5}$ & 
	 $S_{2,4}$ &
	 $S_{2,6}$ &
\\
\hline 
	  $S_{2,6}$ & 
	 $S_{2,4}$ &
	 $S_{2,5}$ &
	 $S_{2,7}$ &
	 $S_{2,8}$ &
\\
\hline 
	  $S_{2,7}$ & 
	 $S_{2,6}$ &
	 $S_{2,8}$ &
\\
\hline 
	  $S_{2,8}$ & 
	 $S_{2,10}$ &
	 $S_{2,6}$ &
	 $S_{2,7}$ &
	 $S_{2,9}$ &
\\
\hline 
	  $S_{2,9}$ & 
	 $S_{2,10}$ &
	 $S_{2,8}$ &
\\

\end{tabular}

\quad
\renewcommand{\arraystretch}{1.25}
\begin{tabular}[h]{c|lllll}

	 Vertex & \multicolumn{5}{c}{Adjacencies}\\

\hline 
	  $S_{2,10}$ & 
	 $S_{1,0\equiv2,12}$ &
	 $S_{2,11}$ &
	 $S_{2,8}$ &
	 $S_{2,9}$ &
\\
\hline 
	  $S_{2,11}$ & 
	 $S_{1,0\equiv2,12}$ &
	 $S_{2,10}$ &
\\

\hline 
	  $S_{2,15}$ & 
	 $S_{1,1\equiv2,14}$ &
	 $S_{2,16}$ &
\\
\hline 
	  $S_{2,16}$ & 
	 $S_{1,1\equiv2,14}$ &
	 $S_{2,18}$ &
	 $S_{2,17}$ &
	 $S_{2,15}$ &
\\
\hline 
	  $S_{2,17}$ & 
	 $S_{2,18}$ &
	 $S_{2,16}$ &
\\
\hline 
	  $S_{2,18}$ & 
	 $S_{2,20}$ &
	 $S_{2,16}$ &
	 $S_{2,17}$ &
	 $S_{2,19}$ &
\\
\hline 
	  $S_{2,19}$ & 
	 $S_{2,18}$ &
	 $S_{2,20}$ &
\\\hline 
	  $S_{2,20}$ & 
	 $S_{2,22}$ &
	 $S_{2,21}$ &
	 $S_{2,19}$ &
	 $S_{2,18}$ &
\\
\hline 
	  $S_{2,21}$ & 
	 $S_{2,22}$ &
	 $S_{2,20}$ &
\\
\hline 
	  $S_{2,22}$ & 
	 $S_{2,0\equiv3,12}$ &
	 $S_{2,21}$ &
	 $S_{2,20}$ &
	 $S_{2,23}$ &
\\
\hline 
	  $S_{2,23}$ & 
	 $S_{2,0\equiv3,12}$ &
	 $S_{2,22}$ &
\\

\hline 
	  $S_{3,0\equiv1,12}$ & 
	 $S_{3,1\equiv1,14}$ &
	 $S_{3,2\equiv1,13}$ &
	 $S_{1,11}$ &
	 $S_{1,10}$ &
\\ &
	 $S_{3,22}$ &
	 $S_{3,23}$ &
\\

\hline 
	  $S_{3,1\equiv1,14}$ & 
	 $S_{3,0\equiv1,12}$ &
	 $S_{3,2\equiv1,13}$ &
	 $S_{1,16}$ &
	 $S_{1,15}$ &
\\

\hline 
	  $S_{3,2\equiv1,13}$ & 
	 $S_{3,1\equiv1,14}$ &
	 $S_{3,0\equiv1,12}$ &
	 $S_{3,4}$ &
	 $S_{3,3}$ &
\\
\hline 
	  $S_{3,3}$ & 
	 $S_{3,2\equiv1,13}$ &
	 $S_{3,4}$ &
\\
\hline 
	  $S_{3,4}$ & 
	 $S_{3,2\equiv1,13}$ &
	 $S_{3,5}$ &
	 $S_{3,6}$ &
	 $S_{3,3}$ &
\\
\hline 
	  $S_{3,5}$ & 
	 $S_{3,4}$ &
	 $S_{3,6}$ &
\\
\hline 
	  $S_{3,6}$ & 
	 $S_{3,5}$ &
	 $S_{3,4}$ &
	 $S_{3,7}$ &
	 $S_{3,8}$ &
\\
\hline 
	  $S_{3,7}$ & 
	 $S_{3,6}$ &
	 $S_{3,8}$ &
\\
\hline 
	  $S_{3,8}$ & 
	 $S_{3,10}$ &
	 $S_{3,7}$ &
	 $S_{3,6}$ &
	 $S_{3,9}$ &
\\
\hline 
	  $S_{3,9}$ & 
	 $S_{3,10}$ &
	 $S_{3,8}$ &

\\

\hline 
	  $S_{3,10}$ & 
	 $S_{2,0\equiv3,12}$ &
	 $S_{3,11}$ &
	 $S_{3,8}$ &
	 $S_{3,9}$ &
\\
\hline 
	  $S_{3,11}$ & 
	 $S_{2,0\equiv3,12}$ &
	 $S_{3,10}$ &
\\

\hline 
	  $S_{3,15}$ & 
	 $S_{2,1\equiv3,14}$ &
	 $S_{3,16}$ &
\\
\hline 
	  $S_{3,16}$ & 
	 $S_{2,1\equiv3,14}$ &
	 $S_{3,18}$ &
	 $S_{3,15}$ &
	 $S_{3,17}$ &
\\
\hline 
	  $S_{3,17}$ & 
	 $S_{3,18}$ &
	 $S_{3,16}$ &
\\
\hline 
	  $S_{3,18}$ & 
	 $S_{3,20}$ &
	 $S_{3,17}$ &
	 $S_{3,16}$ &
	 $S_{3,19}$ &
\\
\hline 
	  $S_{3,19}$ & 
	 $S_{3,18}$ &
	 $S_{3,20}$ &
\\\hline 
	  $S_{3,20}$ & 
	 $S_{3,18}$ &
	 $S_{3,21}$ &
	 $S_{3,22}$ &
	 $S_{3,19}$ &
\\
\hline 
	  $S_{3,21}$ & 
	 $S_{3,20}$ &
	 $S_{3,22}$ &
\\
\hline 
	  $S_{3,22}$ & 
	 $S_{3,0\equiv1,12}$ &
	 $S_{3,20}$ &
	 $S_{3,21}$ &
	 $S_{3,23}$ &
\\
\hline 
	  $S_{3,23}$ & 
	 $S_{3,0\equiv1,12}$ &
	 $S_{3,22}$ &

\end{tabular}

}
\caption{\label{tbl:ClauseGadget}The vertex-labeled clause gadget}
\end{table}

\begin{table}[p]
\small
\centering

\begin{tabular}[h]{c|rlrlrlrlrlr}
	 Vertex & \multicolumn{11}{c}{Adjacencies (Labels)}\\
\hline 
	  $0$ & 
	 $1$ & ($S_{1,1\equiv2,14}$) & 
	 $5$ & ($S_{1,11}$) & 
	 $6$ & ($S_{1,22}$) & 
	 $7$ & ($S_{1,0\equiv2,12}$) & 
	 $14$ & ($S_{2,5}$) & 
\\ &
	 $19$ & ($S_{2,4}$) & 
	 $20$ & ($S_{2,6}$) & 
\\
\hline 
	  $1$ & 
	 $0$ & ($S_{1,1\equiv2,14}$) & 
	 $2$ & ($S_{1,3}$) & 
	 $7$ & ($S_{1,0\equiv2,12}$) & 
	 $8$ & ($S_{1,2\equiv2,13}$) & 
	 $15$ & ($S_{2,8}$) & 
\\ &
	 $16$ & ($S_{2,9}$) & 
	 $20$ & ($S_{2,7}$) & 
\\
\hline 
	  $2$ & 
	 $8$ & ($S_{1,2\equiv2,13}$) & 
	 $1$ & ($S_{1,3}$) & 
	 $3$ & ($S_{1,5}$) & 
	 $9$ & ($S_{1,4}$) & 
\\
\hline 
	  $3$ & 
	 $2$ & ($S_{1,5}$) & 
	 $4$ & ($S_{1,7}$) & 
	 $9$ & ($S_{1,4}$) & 
	 $10$ & ($S_{1,6}$) & 
	 $21$ & ($S_{3,1\equiv1,14}$) & 
\\ &
	 $24$ & ($S_{3,0\equiv1,12}$) & 
\\
\hline 
	  $4$ & 
	 $3$ & ($S_{1,7}$) & 
	 $5$ & ($S_{1,9}$) & 
	 $10$ & ($S_{1,6}$) & 
	 $11$ & ($S_{1,8}$) & 
	 $22$ & ($S_{3,5}$) & 
\\ &
	 $25$ & ($S_{3,4}$) & 
\\
\hline 
	  $5$ & 
	 $0$ & ($S_{1,11}$) & 
	 $11$ & ($S_{1,8}$) & 
	 $4$ & ($S_{1,9}$) & 
	 $6$ & ($S_{1,10}$) & 
\\
\hline 
	  $6$ & 
	 $0$ & ($S_{1,22}$) & 
	 $11$ & ($S_{1,9}$) & 
	 $5$ & ($S_{1,10}$) & 
	 $7$ & ($S_{1,23}$) & 
\\
\hline 
	  $7$ & 
	 $0$ & ($S_{1,0\equiv2,12}$) & 
	 $1$ & ($S_{1,0\equiv2,12}$) & 
	 $6$ & ($S_{1,23}$) & 
	 $8$ & ($S_{1,1\equiv2,14}$) & 
	 $15$ & ($S_{2,9}$) & 
\\ &
	 $20$ & ($S_{2,6}$) & 
\\
\hline 
	  $8$ & 
	 $1$ & ($S_{1,2\equiv2,13}$) & 
	 $2$ & ($S_{1,2\equiv2,13}$) & 
	 $9$ & ($S_{1,3}$) & 
	 $7$ & ($S_{1,1\equiv2,14}$) & 
\\
\hline 
	  $9$ & 
	 $8$ & ($S_{1,3}$) & 
	 $2$ & ($S_{1,4}$) & 
	 $3$ & ($S_{1,4}$) & 
	 $10$ & ($S_{1,5}$) & 
\\
\hline 
	  $10$ & 
	 $3$ & ($S_{1,6}$) & 
	 $4$ & ($S_{1,6}$) & 
	 $9$ & ($S_{1,5}$) & 
	 $11$ & ($S_{1,7}$) & 
	 $24$ & ($S_{3,1\equiv1,14}$) & 
\\ &
	 $25$ & ($S_{3,3}$) & 
\\
\hline 
	  $11$ & 
	 $10$ & ($S_{1,7}$) & 
	 $4$ & ($S_{1,8}$) & 
	 $5$ & ($S_{1,8}$) & 
	 $6$ & ($S_{1,9}$) & 
\\
\hline 
	  $12$ & 
	 $13$ & ($S_{2,1\equiv3,14}$) & 
	 $15$ & ($S_{2,11}$) & 
	 $16$ & ($S_{2,22}$) & 
	 $17$ & ($S_{2,0\equiv3,12}$) & 
	 $18$ & ($S_{3,9}$) & 
\\ &
	 $22$ & ($S_{3,7}$) & 
	 $26$ & ($S_{3,6}$) & 
\\
\hline 
	  $13$ & 
	 $12$ & ($S_{2,1\equiv3,14}$) & 
	 $14$ & ($S_{2,3}$) & 
	 $17$ & ($S_{2,0\equiv3,12}$) & 
	 $18$ & ($S_{2,2\equiv3,13}$) & 
	 $23$ & ($S_{3,9}$) & 
\\ &
	 $26$ & ($S_{3,7}$) & 
\\
\hline 
	  $14$ & 
	 $0$ & ($S_{2,5}$) & 
	 $18$ & ($S_{2,2\equiv3,13}$) & 
	 $19$ & ($S_{2,4}$) & 
	 $13$ & ($S_{2,3}$) & 
\\
\hline 
	  $15$ & 
	 $16$ & ($S_{2,10}$) & 
	 $1$ & ($S_{2,8}$) & 
	 $12$ & ($S_{2,11}$) & 
	 $7$ & ($S_{2,9}$) & 
\\
\hline 
	  $16$ & 
	 $1$ & ($S_{2,9}$) & 
	 $12$ & ($S_{2,22}$) & 
	 $17$ & ($S_{2,23}$) & 
	 $15$ & ($S_{2,10}$) & 
\\
\hline 
	  $17$ & 
	 $16$ & ($S_{2,23}$) & 
	 $18$ & ($S_{2,1\equiv3,14}$) & 
	 $12$ & ($S_{2,0\equiv3,12}$) & 
	 $13$ & ($S_{2,0\equiv3,12}$) & 
\\
\hline 
	  $18$ & 
	 $12$ & ($S_{3,9}$) & 
	 $13$ & ($S_{2,2\equiv3,13}$) & 
	 $14$ & ($S_{2,2\equiv3,13}$) & 
	 $17$ & ($S_{2,1\equiv3,14}$) & 
	 $19$ & ($S_{2,3}$) & 
\\ &
	 $21$ & ($S_{3,11}$) & 
	 $23$ & ($S_{3,10}$) & 
\\
\hline 
	  $19$ & 
	 $0$ & ($S_{2,4}$) & 
	 $18$ & ($S_{2,3}$) & 
	 $20$ & ($S_{2,5}$) & 
	 $14$ & ($S_{2,4}$) & 
\\
\hline 
	  $20$ & 
	 $0$ & ($S_{2,6}$) & 
	 $1$ & ($S_{2,7}$) & 
	 $19$ & ($S_{2,5}$) & 
	 $7$ & ($S_{2,6}$) & 
\\
\hline 
	  $21$ & 
	 $24$ & ($S_{3,0\equiv1,12}$) & 
	 $18$ & ($S_{3,11}$) & 
	 $3$ & ($S_{3,1\equiv1,14}$) & 
	 $23$ & ($S_{3,22}$) & 
\\
\hline 
	  $22$ & 
	 $12$ & ($S_{3,7}$) & 
	 $26$ & ($S_{3,6}$) & 
	 $4$ & ($S_{3,5}$) & 
	 $25$ & ($S_{3,4}$) & 
\\
\hline 
	  $23$ & 
	 $24$ & ($S_{3,23}$) & 
	 $18$ & ($S_{3,10}$) & 
	 $13$ & ($S_{3,9}$) & 
	 $21$ & ($S_{3,22}$) & 
\\
\hline 
	  $24$ & 
	 $10$ & ($S_{3,1\equiv1,14}$) & 
	 $3$ & ($S_{3,0\equiv1,12}$) & 
	 $21$ & ($S_{3,0\equiv1,12}$) & 
	 $23$ & ($S_{3,23}$) & 
\\
\hline 
	  $25$ & 
	 $10$ & ($S_{3,3}$) & 
	 $4$ & ($S_{3,4}$) & 
	 $26$ & ($S_{3,5}$) & 
	 $22$ & ($S_{3,4}$) & 
\\
\hline 
	  $26$ & 
	 $25$ & ($S_{3,5}$) & 
	 $12$ & ($S_{3,6}$) & 
	 $13$ & ($S_{3,7}$) & 
	 $22$ & ($S_{3,6}$) & 
\\
 \end{tabular}

\caption{\label{tbl:0wheels}The edge-labeled preimage of a clause with zero wheels}
\end{table}

\begin{table}[p]
\small
\centering

\begin{tabular}[h]{c|rlrlrlrlrlr}
	Vertex & \multicolumn{11}{c}{Adjacencies (Labels)}\\
\hline 
	  $0$ & 
	 $1$ & ($S_{1,0\equiv2,12}$) & 
	 $2$ & ($S_{1,2\equiv2,13}$) & 
	 $3$ & ($S_{1,4}$) & 
	 $4$ & ($S_{1,6}$) & 
	 $5$ & ($S_{1,8}$) & 
\\ &
	 $6$ & ($S_{1,10}$) & 
	 $7$ & ($S_{1,12}$) & 
	 $8$ & ($S_{1,14}$) & 
	 $9$ & ($S_{1,16}$) & 
	 $10$ & ($S_{1,18}$) & 
\\ &
	 $11$ & ($S_{1,20}$) & 
	 $12$ & ($S_{1,22}$) & 
	 $16$ & ($S_{2,9}$) & 
	 $21$ & ($S_{2,8}$) & 
	 $22$ & ($S_{3,3}$) & 
\\ &
	 $25$ & ($S_{3,2\equiv1,13}$) & 
\\
\hline 
	  $1$ & 
	 $0$ & ($S_{1,0\equiv2,12}$) & 
	 $20$ & ($S_{2,4}$) & 
	 $2$ & ($S_{1,1\equiv2,14}$) & 
	 $12$ & ($S_{1,23}$) & 
	 $15$ & ($S_{2,5}$) & 
\\
\hline 
	  $2$ & 
	 $0$ & ($S_{1,2\equiv2,13}$) & 
	 $1$ & ($S_{1,1\equiv2,14}$) & 
	 $3$ & ($S_{1,3}$) & 
	 $20$ & ($S_{2,5}$) & 
	 $21$ & ($S_{2,7}$) & 
\\
\hline 
	  $3$ & 
	 $0$ & ($S_{1,4}$) & 
	 $2$ & ($S_{1,3}$) & 
	 $4$ & ($S_{1,5}$) & 
\\
\hline 
	  $4$ & 
	 $0$ & ($S_{1,6}$) & 
	 $3$ & ($S_{1,5}$) & 
	 $5$ & ($S_{1,7}$) & 
\\
\hline 
	  $5$ & 
	 $0$ & ($S_{1,8}$) & 
	 $4$ & ($S_{1,7}$) & 
	 $6$ & ($S_{1,9}$) & 
\\
\hline 
	  $6$ & 
	 $0$ & ($S_{1,10}$) & 
	 $5$ & ($S_{1,9}$) & 
	 $7$ & ($S_{1,11}$) & 
\\
\hline 
	  $7$ & 
	 $0$ & ($S_{1,12}$) & 
	 $8$ & ($S_{1,13}$) & 
	 $24$ & ($S_{3,22}$) & 
	 $6$ & ($S_{1,11}$) & 
	 $23$ & ($S_{3,11}$) & 
\\
\hline 
	  $8$ & 
	 $0$ & ($S_{1,14}$) & 
	 $9$ & ($S_{1,15}$) & 
	 $25$ & ($S_{3,1\equiv1,14}$) & 
	 $24$ & ($S_{3,23}$) & 
	 $7$ & ($S_{1,13}$) & 
\\
\hline 
	  $9$ & 
	 $0$ & ($S_{1,16}$) & 
	 $8$ & ($S_{1,15}$) & 
	 $10$ & ($S_{1,17}$) & 
\\
\hline 
	  $10$ & 
	 $0$ & ($S_{1,18}$) & 
	 $9$ & ($S_{1,17}$) & 
	 $11$ & ($S_{1,19}$) & 
\\
\hline 
	  $11$ & 
	 $0$ & ($S_{1,20}$) & 
	 $10$ & ($S_{1,19}$) & 
	 $12$ & ($S_{1,21}$) & 
\\
\hline 
	  $12$ & 
	 $0$ & ($S_{1,22}$) & 
	 $1$ & ($S_{1,23}$) & 
	 $11$ & ($S_{1,21}$) & 
\\
\hline 
	  $13$ & 
	 $14$ & ($S_{2,1\equiv3,14}$) & 
	 $16$ & ($S_{2,11}$) & 
	 $17$ & ($S_{2,22}$) & 
	 $18$ & ($S_{2,0\equiv3,12}$) & 
	 $22$ & ($S_{3,5}$) & 
\\ &
	 $26$ & ($S_{3,4}$) & 
\\
\hline 
	  $14$ & 
	 $13$ & ($S_{2,1\equiv3,14}$) & 
	 $15$ & ($S_{2,3}$) & 
	 $18$ & ($S_{2,0\equiv3,12}$) & 
	 $19$ & ($S_{2,2\equiv3,13}$) & 
	 $26$ & ($S_{3,5}$) & 
\\ &
	 $27$ & ($S_{3,7}$) & 
\\
\hline 
	  $15$ & 
	 $1$ & ($S_{2,5}$) & 
	 $19$ & ($S_{2,2\equiv3,13}$) & 
	 $20$ & ($S_{2,4}$) & 
	 $14$ & ($S_{2,3}$) & 
\\
\hline 
	  $16$ & 
	 $0$ & ($S_{2,9}$) & 
	 $17$ & ($S_{2,10}$) & 
	 $13$ & ($S_{2,11}$) & 
	 $21$ & ($S_{2,8}$) & 
\\
\hline 
	  $17$ & 
	 $16$ & ($S_{2,10}$) & 
	 $18$ & ($S_{2,23}$) & 
	 $13$ & ($S_{2,22}$) & 
	 $21$ & ($S_{2,9}$) & 
\\
\hline 
	  $18$ & 
	 $13$ & ($S_{2,0\equiv3,12}$) & 
	 $14$ & ($S_{2,0\equiv3,12}$) & 
	 $17$ & ($S_{2,23}$) & 
	 $19$ & ($S_{2,1\equiv3,14}$) & 
	 $23$ & ($S_{3,9}$) & 
\\ &
	 $27$ & ($S_{3,8}$) & 
\\
\hline 
	  $19$ & 
	 $18$ & ($S_{2,1\equiv3,14}$) & 
	 $20$ & ($S_{2,3}$) & 
	 $14$ & ($S_{2,2\equiv3,13}$) & 
	 $15$ & ($S_{2,2\equiv3,13}$) & 
\\
\hline 
	  $20$ & 
	 $1$ & ($S_{2,4}$) & 
	 $2$ & ($S_{2,5}$) & 
	 $19$ & ($S_{2,3}$) & 
	 $15$ & ($S_{2,4}$) & 
\\
\hline 
	  $21$ & 
	 $0$ & ($S_{2,8}$) & 
	 $16$ & ($S_{2,8}$) & 
	 $2$ & ($S_{2,7}$) & 
	 $17$ & ($S_{2,9}$) & 
\\
\hline 
	  $22$ & 
	 $0$ & ($S_{3,3}$) & 
	 $25$ & ($S_{3,2\equiv1,13}$) & 
	 $26$ & ($S_{3,4}$) & 
	 $13$ & ($S_{3,5}$) & 
\\
\hline 
	  $23$ & 
	 $24$ & ($S_{3,10}$) & 
	 $18$ & ($S_{3,9}$) & 
	 $27$ & ($S_{3,8}$) & 
	 $7$ & ($S_{3,11}$) & 
\\
\hline 
	  $24$ & 
	 $8$ & ($S_{3,23}$) & 
	 $27$ & ($S_{3,9}$) & 
	 $23$ & ($S_{3,10}$) & 
	 $7$ & ($S_{3,22}$) & 
\\
\hline 
	  $25$ & 
	 $0$ & ($S_{3,2\equiv1,13}$) & 
	 $8$ & ($S_{3,1\equiv1,14}$) & 
	 $26$ & ($S_{3,3}$) & 
	 $22$ & ($S_{3,2\equiv1,13}$) & 
\\
\hline 
	  $26$ & 
	 $25$ & ($S_{3,3}$) & 
	 $22$ & ($S_{3,4}$) & 
	 $13$ & ($S_{3,4}$) & 
	 $14$ & ($S_{3,5}$) & 
\\
\hline 
	  $27$ & 
	 $24$ & ($S_{3,9}$) & 
	 $18$ & ($S_{3,8}$) & 
	 $14$ & ($S_{3,7}$) & 
	 $23$ & ($S_{3,8}$) & 
\\

 \end{tabular}

\caption{\label{tbl:1wheel}The edge-labeled preimage of a clause with one wheel}
\end{table}

\begin{table}[p]
\small
\centering

\begin{tabular}[h]{c|rlrlrlrlrlr}
	Vertex & \multicolumn{11}{c}{Adjacencies (Labels)}\\
\hline 
	  $0$ & 
	 $1$ & ($S_{1,0\equiv2,12}$) & 
	 $2$ & ($S_{1,2\equiv2,13}$) & 
	 $3$ & ($S_{1,4}$) & 
	 $4$ & ($S_{1,6}$) & 
	 $5$ & ($S_{1,8}$) & 
\\ &
	 $6$ & ($S_{1,10}$) & 
	 $7$ & ($S_{1,12}$) & 
	 $8$ & ($S_{1,14}$) & 
	 $9$ & ($S_{1,16}$) & 
	 $10$ & ($S_{1,18}$) & 
\\ &
	 $11$ & ($S_{1,20}$) & 
	 $12$ & ($S_{1,22}$) & 
	 $19$ & ($S_{2,15}$) & 
	 $23$ & ($S_{3,3}$) & 
	 $26$ & ($S_{3,2\equiv1,13}$) & 
\\
\hline 
	  $1$ & 
	 $0$ & ($S_{1,0\equiv2,12}$) & 
	 $2$ & ($S_{1,1\equiv2,14}$) & 
	 $12$ & ($S_{1,23}$) & 
	 $18$ & ($S_{2,11}$) & 
\\
\hline 
	  $2$ & 
	 $0$ & ($S_{1,2\equiv2,13}$) & 
	 $1$ & ($S_{1,1\equiv2,14}$) & 
	 $3$ & ($S_{1,3}$) & 
	 $13$ & ($S_{2,0\equiv3,12}$) & 
	 $14$ & ($S_{2,2\equiv3,13}$) & 
\\ &
	 $15$ & ($S_{2,4}$) & 
	 $16$ & ($S_{2,6}$) & 
	 $17$ & ($S_{2,8}$) & 
	 $18$ & ($S_{2,10}$) & 
	 $19$ & ($S_{2,16}$) & 
\\ &
	 $20$ & ($S_{2,18}$) & 
	 $21$ & ($S_{2,20}$) & 
	 $22$ & ($S_{2,22}$) & 
	 $24$ & ($S_{3,9}$) & 
	 $28$ & ($S_{3,8}$) & 
\\
\hline 
	  $3$ & 
	 $0$ & ($S_{1,4}$) & 
	 $2$ & ($S_{1,3}$) & 
	 $4$ & ($S_{1,5}$) & 
\\
\hline 
	  $4$ & 
	 $0$ & ($S_{1,6}$) & 
	 $3$ & ($S_{1,5}$) & 
	 $5$ & ($S_{1,7}$) & 
\\
\hline 
	  $5$ & 
	 $0$ & ($S_{1,8}$) & 
	 $4$ & ($S_{1,7}$) & 
	 $6$ & ($S_{1,9}$) & 
\\
\hline 
	  $6$ & 
	 $0$ & ($S_{1,10}$) & 
	 $5$ & ($S_{1,9}$) & 
	 $7$ & ($S_{1,11}$) & 
\\
\hline 
	  $7$ & 
	 $0$ & ($S_{1,12}$) & 
	 $8$ & ($S_{1,13}$) & 
	 $25$ & ($S_{3,22}$) & 
	 $6$ & ($S_{1,11}$) & 
	 $24$ & ($S_{3,11}$) & 
\\
\hline 
	  $8$ & 
	 $0$ & ($S_{1,14}$) & 
	 $9$ & ($S_{1,15}$) & 
	 $26$ & ($S_{3,1\equiv1,14}$) & 
	 $25$ & ($S_{3,23}$) & 
	 $7$ & ($S_{1,13}$) & 
\\
\hline 
	  $9$ & 
	 $0$ & ($S_{1,16}$) & 
	 $8$ & ($S_{1,15}$) & 
	 $10$ & ($S_{1,17}$) & 
\\
\hline 
	  $10$ & 
	 $0$ & ($S_{1,18}$) & 
	 $9$ & ($S_{1,17}$) & 
	 $11$ & ($S_{1,19}$) & 
\\
\hline 
	  $11$ & 
	 $0$ & ($S_{1,20}$) & 
	 $10$ & ($S_{1,19}$) & 
	 $12$ & ($S_{1,21}$) & 
\\
\hline 
	  $12$ & 
	 $0$ & ($S_{1,22}$) & 
	 $1$ & ($S_{1,23}$) & 
	 $11$ & ($S_{1,21}$) & 
\\
\hline 
	  $13$ & 
	 $2$ & ($S_{2,0\equiv3,12}$) & 
	 $27$ & ($S_{3,4}$) & 
	 $22$ & ($S_{2,23}$) & 
	 $14$ & ($S_{2,1\equiv3,14}$) & 
	 $23$ & ($S_{3,5}$) & 
\\
\hline 
	  $14$ & 
	 $2$ & ($S_{2,2\equiv3,13}$) & 
	 $27$ & ($S_{3,5}$) & 
	 $28$ & ($S_{3,7}$) & 
	 $13$ & ($S_{2,1\equiv3,14}$) & 
	 $15$ & ($S_{2,3}$) & 
\\
\hline 
	  $15$ & 
	 $16$ & ($S_{2,5}$) & 
	 $2$ & ($S_{2,4}$) & 
	 $14$ & ($S_{2,3}$) & 
\\
\hline 
	  $16$ & 
	 $17$ & ($S_{2,7}$) & 
	 $2$ & ($S_{2,6}$) & 
	 $15$ & ($S_{2,5}$) & 
\\
\hline 
	  $17$ & 
	 $16$ & ($S_{2,7}$) & 
	 $2$ & ($S_{2,8}$) & 
	 $18$ & ($S_{2,9}$) & 
\\
\hline 
	  $18$ & 
	 $1$ & ($S_{2,11}$) & 
	 $2$ & ($S_{2,10}$) & 
	 $17$ & ($S_{2,9}$) & 
\\
\hline 
	  $19$ & 
	 $0$ & ($S_{2,15}$) & 
	 $2$ & ($S_{2,16}$) & 
	 $20$ & ($S_{2,17}$) & 
\\
\hline 
	  $20$ & 
	 $2$ & ($S_{2,18}$) & 
	 $19$ & ($S_{2,17}$) & 
	 $21$ & ($S_{2,19}$) & 
\\
\hline 
	  $21$ & 
	 $2$ & ($S_{2,20}$) & 
	 $20$ & ($S_{2,19}$) & 
	 $22$ & ($S_{2,21}$) & 
\\
\hline 
	  $22$ & 
	 $2$ & ($S_{2,22}$) & 
	 $21$ & ($S_{2,21}$) & 
	 $13$ & ($S_{2,23}$) & 
\\
\hline 
	  $23$ & 
	 $0$ & ($S_{3,3}$) & 
	 $26$ & ($S_{3,2\equiv1,13}$) & 
	 $27$ & ($S_{3,4}$) & 
	 $13$ & ($S_{3,5}$) & 
\\
\hline 
	  $24$ & 
	 $25$ & ($S_{3,10}$) & 
	 $2$ & ($S_{3,9}$) & 
	 $28$ & ($S_{3,8}$) & 
	 $7$ & ($S_{3,11}$) & 
\\
\hline 
	  $25$ & 
	 $8$ & ($S_{3,23}$) & 
	 $24$ & ($S_{3,10}$) & 
	 $28$ & ($S_{3,9}$) & 
	 $7$ & ($S_{3,22}$) & 
\\
\hline 
	  $26$ & 
	 $0$ & ($S_{3,2\equiv1,13}$) & 
	 $8$ & ($S_{3,1\equiv1,14}$) & 
	 $27$ & ($S_{3,3}$) & 
	 $23$ & ($S_{3,2\equiv1,13}$) & 
\\
\hline 
	  $27$ & 
	 $26$ & ($S_{3,3}$) & 
	 $13$ & ($S_{3,4}$) & 
	 $14$ & ($S_{3,5}$) & 
	 $23$ & ($S_{3,4}$) & 
\\
\hline 
	  $28$ & 
	 $24$ & ($S_{3,8}$) & 
	 $25$ & ($S_{3,9}$) & 
	 $2$ & ($S_{3,8}$) & 
	 $14$ & ($S_{3,7}$) & 
\\

 \end{tabular}

\caption{\label{tbl:2wheels}The edge-labeled preimage of a clause with two wheels}
\end{table}

\end{document}